\documentclass[11pt,leqno]{amsart}
\usepackage{amscd}
\usepackage{color}
\usepackage[symbol]{footmisc}
\usepackage{amssymb}
\usepackage{amsfonts}
\usepackage{latexsym}
\usepackage{verbatim}
\usepackage{bbm}

\newcommand{\R}{\mathbb{R}}
\newcommand{\C}{\mathbb{C}}
\renewcommand{\H}{\mathbb{H}}
\renewcommand{\epsilon}{\varepsilon}
\renewcommand{\theta}{\vartheta}
\renewcommand{\phi}{\varphi}
\renewcommand{\Re}{\mathrm{Re}}

\theoremstyle{plain}
\newtheorem{teor}{Theorem}
\newtheorem{prop}[teor]{Proposition}
\newtheorem{lem}[teor]{Lemma}
\newtheorem{cor}[teor]{Corollary}

\theoremstyle{definition}
\newtheorem*{oss}{Remark}

\theoremstyle{definition}
\newtheorem{defn}[teor]{Definition}

\title[Fully non-linear elliptic equations on hyperhermitian manifolds]{Fully non-linear elliptic equations on compact manifolds with a flat hyperk\"ahler metric}

\begin{document}
	
%\thanks{This work was supported by GNSAGA of INdAM}

	\address{Dipartimento di Matematica G. Peano \\ Universit\`a di Torino\\
		Via Carlo Alberto 10\\
		10123 Torino\\ Italy}
	\email{giovanni.gentili@unito.it}
	
	\address{Jiaogen Zhang, School of Mathematical Sciences, University of Science and Technology of China, Hefei 230026, People's Republic of China }
	\email{zjgmath@ustc.edu.cn}
	
	\author{Giovanni Gentili and Jiaogen Zhang}
	
	\date{\today}

	\maketitle
	\begin{abstract}
		Mainly motivated by a conjecture of Alesker and Verbitsky, we study a class of fully non-linear elliptic equations on certain compact hyperhermitian manifolds.  By adapting the approach of Sz\'{e}kelyhidi \cite{Sze} to the hypercomplex setting, we prove some a priori estimates for solutions to such equations under the assumption of existence of $\mathcal{C}$-subsolutions. In the estimate of the quaternionic Laplacian,  we need to further assume the existence of a flat hyperk\"ahler metric. As an application of our results we prove that the quaternionic analogue of the Hessian equation and Monge-Amp\`ere equation for $(n-1)$-plurisubharmonic functions can always be solved on compact flat hyperk\"ahler manifolds.
	\end{abstract}
	
	\section{Introduction}
	A {\em hypercomplex manifold} is a smooth manifold $M$ of real dimension $ 4n $ equipped with a triple of complex structures $(I,J,K)$ satisfying the quaternionic relations
	\[
	IJ=-JI=K\,.
	\]
	A Riemannian metric $ g $ on a hypercomplex manifold $ (M,I,J,K) $ is said to be {\em hyperhermitian} if it is Hermitian with respect to each of $ I,J,K $. Any hyperhermitian metric induces a $2$-form 
	$$
	\Omega_{0}=\omega_J+i\omega_K
	$$ 
	where $\omega_J$ and $\omega_K$ are the fundamental forms of $(g,J)$ and $(g,K)$ respectively. The form $\Omega_{0}$ is of type $(2,0)$ with respect to $I$, satisfies the {\em $q$-real} condition $J\Omega_{0}=\overline{\Omega_{0}} $ (here $ J $ acts on $ \Omega_0 $ as $J\Omega_0(\cdot,\cdot)=\Omega_{0}(J\cdot,J\cdot)$) and is {\em positive} in the sense that $\Omega_0(X,JX)>0$ for every nowhere vanishing real vector field on $M$. $(M,I,J,K,g)$ is hyperk\"ahler if and only if $\Omega_0$ is closed, while it is called HKT ({\em hyperk\"ahler with torsion}) if $\Omega_0$ is $\partial$-closed, where $\partial$ is taken with respect to $I$. The geometry of HKT manifolds is widely studied in the literature (see e.g. \cite{Alesker-Verbitsky (2006),Banos,Barberis-Fino,GF,GLV,Grantcharov-Poon (2000),Howe-Papadopoulos (1996),Ivanov,Lejmi-Weber,SThesis,Swann,Verbitsky (2002),Verbitsky (2007),Verbitsky (2009)} and the reference therein). \\
	
	Fix a $q$-real $(2,0)$ form $\Omega$, a smooth map $\varphi\colon M\to \R$ on a hyperhermitian manifold $(M,I,J,K,g)$ is called {\em quaternionic $ \Omega $-plurisubharmonic} if 
	$$
	\Omega_\phi:=\Omega+\partial\partial_J\varphi \mbox{ is positive}\,,
	$$
	where 
	\[
	\partial_J:=J^{-1}\bar \partial J
	\]
	is the {\em twisted differential operator} introduced by Verbitsky \cite{Verbitsky (2002)}, being $\bar \partial $ the conjugate of $\partial$. Animated by the study of ``canonical'' HKT metrics, in analogy to the Calabi conjecture \cite{Calabi} proved by Yau in \cite{yau}, Alesker and Verbitsky proposed in \cite{Alesker-Verbitsky (2010)} to study the {\em quaternionic Monge-Amp\`ere equation}:
	\begin{equation}\label{qMA}
		\Omega_\varphi^n=b\,{\rm e}^H\Omega^n_0
	\end{equation}
	on a compact HKT manifold, where $H\in C^{\infty}(M,\R)$ is given, while $(\varphi,b)\in C^{\infty}(M,\R)\times \R_+$ is the unknown. Even if the solvability of the quaternionic  Monge-Amp\`ere equation is still an open problem in its general form, several partial results are available in the literature \cite{Alesker (2013),Alesker-Shelukhin (2013),Alesker-Shelukhin (2017),Alesker-Verbitsky (2010),BGV,DinewSroka,GentiliVezzoni,GV,Sroka,Z}. A nice geometric application of the solvability of equation \eqref{qMA} is the existence of a unique balanced metric $\tilde g$ on a compact HKT manifold $(M,I,J,K,g)$ with holomorphically trivial canonical bundle with respect to $I$ such that the form $\tilde \Omega$ induced by $\tilde g$ belongs to the class $\{\Omega+\partial \partial_J \varphi\}$ (see \cite{Verbitsky (2009)}).  From this point of view, equation \eqref{qMA} is the ``quaternionic counterpart'' of the complex Monge-Amp\`ere equation in K\"ahler geometry and balanced HKT metrics play the role that Calabi-Yau metrics play in K\"ahler geometry.

	\medskip 
	Following the parallelism between Hermitian and hyperhermitian geometry it is quite natural to enlarge the study of the quaternionic Monge-Amp\`ere equation to
	a general set of fully non-linear elliptic equations on hypercomplex manifolds. Here we adapt the description given by Sz\'{e}kelyhidi in \cite{Sze} to the hypercomplex setting.
	
	\medskip
	In the current paper we consider hypercomplex manifolds $(M,I,J,K)$ which are locally isomorphic to $\mathbb H^n$. Unlike complex manifolds, where the integrability of the complex structure guarantees that every point has a neighborhood biholomorphic to an open subset of $\C^n$, for hypercomplex manifolds the integrability of the hypercomplex structure is not enough to ensure that $(M,I,J,K)$ is locally isomorphic to the standard flat space. These manifolds were first introduced by Sommese in \cite{S} and are today called {\em locally flat} since they can be characterized as hypercomplex manifolds having the curvature of the Obata connection \cite{Obata (1956)} identically zero. We recall that the Obata connection $ \nabla$ is the unique torsion free connection on a hypercomplex manifold $(M,I,J,K)$ that preserves the hypercomplex structure, i.e.
	\[
	\nabla I=\nabla J= \nabla K =0\,.
	\]
	%Let $ (M,I,J,K,g) $ be a compact hyperhermitian manifold. The hypercomplex structure $ (I,J,K) $ is called \emph{locally flat} if every point of $ M $ has a neighborhood isomorphic to $ \H^n $. This means that, unless we assume this condition, we cannot talk about \emph{quaternionic coordinates}. This is a major difference with respect to complex and real manifold which always admit neighborhoods isomorphic to open subsets of $ \C^n $ and $ \R^n $ respectively. Manifolds of this kind were studied by Sommese in \cite{S} and called \emph{quaternionic manifolds}. 
	
	On  a  {\em locally flat} hypercomplex manifold $ (M,I,J,K) $,  we can locally find  real coordinates $\{x_p^r\}$, $p=0,1,2,3$, $r=1,\dots n$, such that $(I,J,K)$
	takes the standard form. Setting 
	$$
	q^{r}:=\sum_{p=0}^{3}x_p^r\,e_p\,,
	$$
	where, in order to simplify the notation, we denote the unit quaternions $1,i,j,k$ with $e_0,e_1,e_2,e_3$, we can define the $\mathbb H^{n}$-valued function $(q^1,\dots,q^n)$, which we refer to as {\em quaternionic coordinates}. We can then introduce the {\em quaternionic derivatives} $\partial_{q^r}$ and $\partial_{\bar q^r}$ acting on a smooth $ \H $-valued function $u$ as 
	$$
	\partial_{q^r}u:=\partial_{x_0^r}u\, e_0-\sum_{i=1}^3 \partial_{x_i^r}u\, e_i \,,\quad 
	\partial_{	\bar q^r}u:=\sum_{i=0}^3e_{i}\, \partial_{x_i^r}u \,.
	$$
	The operators $\partial_{q^r}$ and $\partial_{\bar q^s}$ commute, but they do not satisfy the Leibniz rule.  Using the vector fields  $\partial_{q^r}$ and $\partial_{\bar q^r}$ we can locally regard every q-real $(2,0)$-form $ \Omega$ on $M$ as a hyperhermitian matrix $ (\Omega_{\bar r s})$, i.e. as a $n \times n$ quaternionic matrix lying in $ \mathrm{Hyp}(n,\H)=\{ H \in \H^{n,n} \mid H= H^* \} $, where $ H^*={^t\bar H} $. Moreover, for a smooth real-valued function $\varphi$ on $M$,  the matrix associated to $\Omega_\phi=\Omega+\partial\partial_J \varphi$ is $(\Omega^\phi_{\bar rs})=(\Omega_{\bar r s}+\tfrac14 \partial_{\bar q^r}\partial_{q^s}\varphi)$. The matrix $\mathrm{Hess}_\H \phi:=(\varphi_{\bar rs})=(\tfrac14 \partial_{\bar q^r}\partial_{q^s}\varphi)$ is usually called the {\em quaternionic Hessian} of $\varphi$.
	
	\medskip 
	Now we can describe the class of equations we take into account in the present paper.\\
	Let $(M,I,J,K,g)$ be a compact locally flat hyperhermitian manifold and let $\Omega$ be a fixed q-real $(2,0)$-form on $M$ ($\Omega$ is not necessarily the $(2,0)$-form induced by $g$). For  a smooth real function $\varphi$ on $M$ let $\Omega_{\varphi}:=\Omega+\partial\partial_J\varphi$ and $ A^r_s=g^{\bar j r}\Omega^\phi_{ \bar j s}$.  
	The matrix $(A^r_s)$ defines a hyperhermitian endomorphism of $ TM $ with respect to the metric $ g $, i.e. $ A=g^{-1}A^* g $. Note that in general, for quaternionic matrices one does not have (right) eigenvalues in the usual sense, rather conjugacy classes of them. However for hyperhermitian matrices there is a single real eigenvalue in each conjugacy class. Therefore, we consider the function $ \lambda \colon \mathrm{Hyp}(n,\H) \to \R^n $ which associates to a matrix $A$ the $ n $-tuple of its eigenvalues $ \lambda(A) $.
	
	We can then consider an equation of the following type 
	\begin{equation}
		\label{eq_main}
		F(A)=h\,,
	\end{equation}
	where $ h\in C^\infty(M,\R) $ is given and $ F(A)=f(\lambda(A)) $ is a smooth symmetric operator of the eigenvalues of $ A $. Here $ f\colon \Gamma \to \R $, where   
	$ \Gamma $ is a proper convex open cone in $ \R^n $ with vertex at the origin which is symmetric (i.e. it is invariant under permutations of the $ \lambda_i $'s) and contains the positive orthant 
	\[
	\Gamma_n=\{ \lambda=(\lambda_1,\dots,\lambda_n) \in \R^n \mid \lambda_i>0,\, i=1,\dots,n \}\,.
	\]
	We further require that $ f\colon \Gamma \to \R $ satisfies the following assumptions:
	\begin{enumerate}
		\itemsep0.2em
		\item[C1)] $ f_i:=\frac{\partial f}{\partial \lambda_i}>0 $ for all $ i=1,\dots,n $ and $ f $ is a concave function.
		\item[C2)] $ \sup_{\partial \Gamma}f<\inf_Mh $, where $ \sup_{\partial \Gamma}f=\sup_{\lambda_0\in \partial \Gamma} \limsup_{\lambda \to \lambda_0} f(\lambda) $.
		\item[C3)] For any $ \sigma <\sup_\Gamma f $ and $ \lambda \in \Gamma $ we have $ \lim_{t\to \infty} f(t\lambda)>\sigma $.
	\end{enumerate}
	
	Assumption C1 ensures that equation \eqref{eq_main} is elliptic when $\phi$ is {\em $\Gamma$-admissible}, i.e.
	\[
	\lambda\left( g^{\bar k r}(\Omega_{\bar k s} +\phi_{\bar k s} \right)\in \Gamma\,.
	\]
	Assumption C2 says that the level sets of $f$ never touch the boundary of $\Gamma$, which also ensures that \eqref{eq_main} is non-degenerate and then uniformly elliptic once we have established the $C^{2}$ estimate.
	
	An analogue framework was firstly considered by Caffarelli, Nirenberg and Spruck \cite{CNS} in $ \R^n$ and later by Li \cite{Li}, Urbas \cite{Urbas}, Guan \cite{Gu,Guan} and Guan and Jiao \cite{GJ} on Riemannian manifolds. Sz\'{e}kelyhidi \cite{Sze} studied this framework in Hermitian Geometry for elliptic equations and Phong and T\^o \cite{PhongTo} for parabolic equations. Sz\'{e}kelyhidi's work has been recently generalized in  \cite{CHZ,HZ} to the almost Hermitian setting. 
	
	\medskip
	Our main result is the following:
	
	\begin{teor}\label{teor_main}
		Let $ (M,I,J,K,g) $ be a compact flat hyperk\"ahler manifold, $ \Omega $ a q-real $ (2,0) $-form, and $ \underline{\phi} $ a $ \mathcal{C} $-subsolution of \eqref{eq_main}. Then there exist $ \alpha \in (0,1) $ and a constant $ C>0 $, depending only on $ (M,I,J,K,g),$  $\Omega,$  $h $ and $ \underline{\phi} $, such that any $\Gamma$-admissible solution $ \phi $ to \eqref{eq_main} with $ \sup_M\phi=0 $ satisfies the estimate
		\[
		\|\phi\|_{C^{2,\alpha}}\leq C\,.
		\]
	\end{teor}
	
	In the above statement by $ \mathcal{C} $-subsolution of \eqref{eq_main} we mean that 
	\[
	\text{for every $ x\in M $ the set }
	\left( \lambda\left(g^{\bar j r}(\Omega_{\bar j s}+\underline \phi_{\bar j s})\right) +\Gamma_n \right)\cap \partial \Gamma^{h(x)} \text{ is bounded}\,,
	\]
	where for any $ \sigma>\sup_{\partial \Gamma} f $, $ \Gamma^\sigma$ denotes the convex superlevel set $ \Gamma^\sigma=\{ \lambda \in \Gamma \mid f(\lambda)>\sigma \} $.
	
	We remark that the assumption of admitting a flat hyperk\"ahler metric in particular implies that $(M,I,J,K)$ is locally flat. 
	
	\medskip As an application of Theorem \ref{teor_main} we first have the solvability of the quaternionic Hessian equation on hyperhermitian manifolds admitting a flat hyperk\"ahler metric.
	
	Let $(M,I,J,K,g,\Omega_0)$ be a compact hyperhermitian manifold where $\Omega_0$ is the $(2,0)$-form induced by $g$, fix $ 1\leq k \leq n $ and let $\Omega$ be a $q$-real $(2,0)$-form which is $k$-positive in the sense that
	\begin{equation}\label{k-pos}
		\frac{\Omega^{i}\wedge \Omega_0^{n-i}}{\Omega_0^n}>0 \qquad \mbox{ for every }i=1,\dots,k\,. 
	\end{equation}
	Let ${\rm QSH}_{k}(M,\Omega)$  be the set of continuous functions $\varphi$ such that $\Omega_{\varphi}$ is a $k$-positive $q$-real $(2,0)$-form in the sense of currents. Then the {\em quaternionic Hessian equation} is defined as 
	\begin{equation}\label{qh}
		\frac{\Omega^{k}_\varphi \wedge \Omega^{n-k}_0}{\Omega^{n}_0}=b\,{\rm e}^{H},\qquad \varphi\in {\rm QSH}_{k}(M,\Omega)\,,
	\end{equation}
	where $ H\in C^\infty(M,\R) $ is the datum and $ (\phi,b)\in C^\infty (M,\R)\times \R_+ $ is the unknown. Equation \eqref{qh} reduces to the quaternionic Monge-Amp\`{e}re equation for $ k=n $ and to the classical Poisson equation for $ k=1 $. Moreover equation \eqref{qh} is the analogue of the real and complex Hessian equations (see, e.g., \cite{Chou-Wang,CHZ,DK,Hou,Hou-Ma-Wu,Jbilou,Kokarev,Li,SW19,Urbas,Wang,Zhang} and the references therein) in the quaternionic setting. The constant $ b $ is uniquely determined by
	\begin{equation*}
		b=\frac{\int_M \Omega^k_\phi\wedge \Omega^{n-k}_0\wedge \bar \Omega_0^n}{\int_M \mathrm{e}^H \Omega^n_0\wedge \bar \Omega_0^n}\,.
	\end{equation*}
	
	Applying Theorem \ref{teor_main} we solve equation \eqref{qh} on compact flat hyperk\"ahler manifolds:
	
	\begin{teor}\label{teor_Hessian}
		Let $(M,I,J,K,g,\Omega_0)$ be a compact flat hyperk\"ahler manifold and $ \Omega $ a q-real $ k $-positive $ (2,0) $-form. Then the quaternionic Hessian equation
		\[
		\frac{\Omega^{k}_\varphi \wedge \Omega^{n-k}_0}{\Omega^{n}_0}=b\,\mathrm{e}^H\,, \qquad\int_M  \varphi\,\Omega^n_0 \wedge\bar \Omega^n_0=0\,, \qquad \varphi\in {\rm QSH}_{k}(M,\Omega)\,,
		\]
		has a unique smooth solution $(\varphi,b)\in C^\infty(M,\R)\times \R_+$ for every $H\in C^{\infty}(M,\R)$.
	\end{teor}
	
	From Theorem \ref{teor_Hessian} we recover as a special case the result of Alesker \cite{Alesker (2013)}, where the quaternionic Monge-Amp\`{e}re equation is solved on compact flat hyperk\"ahler manifolds. We note that during the proof of Theorem \ref{teor_main} the a priori estimates, except for the $C^2$-estimate, are obtained without assuming anything about the closure of $\Omega_0$ and this suggests that it is worth studying the quaternionic Hessian equation on non-HKT hyperhermitian manifolds.
	\medskip
	
	Our second application is the quaternionic  Monge-Amp\`ere equation for $(n-1)$-quaternionic plurisubharmonic functions.
	%Let $U\subseteq M$ be the domain of a coordinate chart on $M$. A function $\phi \in C^2(U,\R)$ is $k$-quaternionic plurisubharmonic ($k$-QPSH for short) if it is subharmonic when restricted to any (right) quaternionic $k$-plane intersecting $U$. A function $\phi \in C^2(U,\R)$ is $(n-1)$-QPSH if and only if the $(2,0)$-form $(\Delta_g \phi) \Omega_0-\partial \partial_J \phi$ is non-negative, where $\Omega_0=\sum_{i=1}^n dz^{2i-1}\wedge dz^{2i}$ is the standard HKT form on $\H^n$ and $\Delta_g$ is the quaternionic Laplacian with respect to the standard Euclidean metric $g$ (see section \ref{C0} for more details). Another equivalent definition is that the sum of any $n-1$ eigenvalues of $\mathrm{Hess}_\H \phi$ is non-negative.
	Let $(M,I,J,K,g,\Omega_0)$ be a compact hyperhermitian manifold and $\Omega_{1}$ be a positive $q$-real $(2,0)$-form. We say that a $C^2$ function $\phi$ on $M$ is $(n-1)$-quaternionic plurisubharmonic with respect to $\Omega_1$ and $\Omega_0$ if the $(2,0)$-form $\Omega_1+\frac{1}{n-1}[(\Delta_g \phi) \Omega_0-\partial \partial_J \phi] $ is pointwise positive, where  $\Delta_g $ is the quaternionic Laplacian with respect to  $g $ (see section \ref{C0} for more details). We also refer to Harvey and Lawson \cite{HL11,HL12} for more general notions of plurisubharmonicity. The {\em quaternionic Monge-Amp\`ere equation for $(n-1)$-quaternionic plurisubharmonic functions} is written as 	
		\begin{equation}\label{n-1 MA}
			\left(\Omega_{1}+\frac{1}{n-1}\big[(\Delta_{g} \phi)\Omega_{0}-\partial\partial_J \phi\big]\right)^{n} = b\,{\rm e}^{H}\Omega_{0}^{n}\,, \qquad \Omega_{1}+\frac{1}{n-1}\big[(\Delta_{g} \phi)\Omega_{0}-\partial\partial_J \phi \big] > 0\,.		
		\end{equation}
	Here the constant $ b $ is uniquely determined by
		\begin{equation*}
			b=\frac{\int_M \big(\Omega_{1}+\frac{1}{n-1}\big[(\Delta_{g} \phi)\Omega_{0}-\partial\partial_J \phi\big]\big)^{n}\wedge \bar \Omega_0^n}{\int_M \mathrm{e}^H \Omega^n_0\wedge \bar \Omega_0^n}\,.
		\end{equation*}
		Equation \eqref{n-1 MA} is the analogue of the complex Monge-Amp\`ere equation for $(n-1)$-plurisubharmonic functions, introduced and studied by Fu-Wang-Wu \cite{FWW10,FWW15}, it is a kind of Monge-Amp\`ere--type equation. More related works can be found in  \cite{CHZ,HZ,TW17,TW19} and the references therein.	
		\begin{teor}\label{teor n-1 MA}
			Let $(M,I,J,K,g,\Omega_0)$ be a compact flat hyperk\"ahler manifold and $ \Omega_{1} $ a q-real positive $ (2,0) $-form. Then there is a unique solution $(\phi,b)\in C^\infty(M,\R)\times \R_+$ to the equation
			\begin{equation}\label{n-1 MA equation 1}
				\begin{cases}
					\ \big(\Omega_{1}+\frac{1}{n-1}\big[(\Delta_{g} \phi)\Omega_{0}-\partial\partial_J \phi\big]\big)^{n} = b\,{\rm e}^{H}\Omega_{0}^{n}\,, \\[2mm]
					\ \Omega_{1}+\frac{1}{n-1}\big[(\Delta_{g} \phi)\Omega_{0}-\partial\partial_J \phi\big] > 0\,, \quad \sup_{M}\phi = 0\,,
				\end{cases}
			\end{equation}
			for every given $H\in C^{\infty}(M,\R)$.
	\end{teor}
	
	From Theorem \ref{teor n-1 MA} we can also obtain Calabi-Yau--type Theorems for quaternionic balanced, quaternionic Gauduchon and quaternionic strongly Gauduchon metrics. We refer the reader to \cite[Table 2]{LW} for the relevant definitions, which are entirely analogous to the complex case.
	
	\begin{cor}\label{cor}
	Let $(M,I,J,K,g,\Omega_0)$ be a compact flat hyperk\"ahler manifold and take a quaternionic balanced (resp. quaternionic Gauduchon, quaternionic strongly Gauduchon) metric with induced $(2,0)$-form $\Omega_2$. Then there is a unique positive constant $b'$ and a unique quaternionic balanced (resp. quaternionic Gauduchon, quaternionic strongly Gauduchon) metric with induced $(2,0)$-form $\tilde \Omega$, such that
	\[
	\tilde \Omega^{n-1}=\Omega_2^{n-1}+\partial \partial_J \phi \wedge \Omega^{n-2}_0\,,
	\]
	for some $\phi \in C^\infty(M,\R)$, and which solves
	\[
	\tilde \Omega^n=b'\, \mathrm{e}^{H'} \Omega^n_0\,,
	\]
	for any given $H'\in C^\infty(M,\R).$
	\end{cor}
	\medskip
	
	The paper is organized as follows: sections \ref{C0}--\ref{C2a} contain the proof of Theorem \ref{teor_main}, while in the last section we prove Theorems \ref{teor_Hessian} and \ref{teor n-1 MA} and Corollary \ref{cor}.
	
	More precisely, in section \ref{C0} we prove the $C^0 $ a priori estimate for solutions to \eqref{eq_main} by using the Alexandroff-Bakelman-Pucci (ABP) method as in \cite{Sze}. Section \ref{Laplacian} deals with the $ C^0 $-estimate for the quaternionic Laplacian in terms of the gradient. This estimate is obtained by bounding the highest eigenvalue of the matrix $A$ and here is where we use the assumption of having a flat hyperk\"ahler metric. 
	The Laplacian estimate is then used to perform the \emph{blow-up analysis} in section \ref{Blow-up} and reduce the gradient bound to the proof of a {\em Liouville-type theorem}, which is given in section \ref{Liouville}. This yields, in particular, a (non-explicit) bound on the quaternionic Laplacian. Finally in section \ref{C2a} we conclude the proof of Theorem \ref{teor_main} applying an Evans-Krylov type theorem \cite{Evans,Krylov} of which we give two proofs, one in the same spirit of \cite{Tosatti et al.}, the other by following an argument of B{\l}ocki \cite{Blo05} as in Alesker \cite{Alesker (2013)}. \medskip 
	
	%\noindent {\bf Data availability statement.} Data sharing not applicable to this article as no datasets were generated or analysed during the current study. \medskip
	
	\noindent {\bf Acknowledgements.} The first author is deeply grateful to his advisor professor Luigi Vezzoni for constant support and many useful observations. He would also like to express his gratitude to Marcin Sroka for useful conversations. The first author is supported by GNSAGA of INdAM. The second author wishes to thank his thesis advisor professor Xi Zhang for his constant supports and encouragements, and he is supported by the National Key R and D Program of China 2020YFA0713100.

	\section{$ C^0 $-estimate}\label{C0}
	The $ C^0 $-estimate for solutions to \eqref{eq_main} is obtained by adapting \cite[Proposition 10]{Sze} to our setting and 
	by using the ABP method. This idea was inspired by the $ C^0 $-bound of B\l{}ocki \cite{Blocki} for the complex Monge-Amp\`{e}re equation.
	
	\medskip
	On a hyperhermitian manifold $(M,I,J,K,g)$, the {\em quaternionic Laplacian} of a real function $\varphi$ is defined by 
	\[
	\Delta_g\phi:=n\frac{\partial \partial_J \phi\wedge \Omega_0^{n-1}}{\Omega_0^n}\,,
	\]
	where $ \Omega_0$ is the $ (2,0) $-form induced by $g$. This is an elliptic second order linear differential operator. Under the assumption of local flatness, by \cite[Lemma 3]{BGV} we have
	\begin{equation*}
		\Delta_g \phi=\Re\,\mathrm{tr}_g(\mathrm{Hess}_\H\phi )=\Re \left(g^{\bar jr} \phi_{\bar j r}\right)\,.
	\end{equation*}
	Consequently, in quaternionic local coordinates, the quaternionic Laplacian is the sum of the eigenvalues of the quaternionic Hessian with respect to $ g $.
	\medskip
	
	Finally, it will be useful to observe that the domain $\Gamma$ of $f$ satisfies
	\begin{equation}\label{cone}
		\Gamma \subseteq \left\{ (\lambda_1,\dots, \lambda_n) \in \R^n \mid \sum_{i=1}^n \lambda_i>0 \right\}\,.
	\end{equation}

	As a preliminary step, we prove an $ L^p $-estimate. From here on, we will always denote with $ C $ a positive constant that only depends on background data and which may change from line to line. 
	\begin{lem}
		Let $ (M,I,J,K,g) $ be a compact locally flat hyperhermitian manifold. If $ \phi $ is a solution to \eqref{eq_main} such that $ \sup_M \phi=0 $, then there exist $ p,C>0 $, depending only on the background data, such that
		\[
		\|\phi\|_{L^p}\leq C\,.
		\]
	\end{lem}
	\begin{proof}
		From \eqref{cone} we have $ \Re\,\mathrm{tr}_g(\Omega_\phi)>0 $, where $ \Omega_\phi =\Omega+\partial \partial_J\phi $, which in turn translates into a lower bound for the quaternionic Laplacian of $ \phi $:
		\begin{equation}\label{eq_in_C0}
			\Delta_g \phi=\Re\, \mathrm{tr}_g(\Omega_\phi)-\Re\,\mathrm{tr}_g(\Omega) \geq -C\,.
		\end{equation}
		An  $L^1 $-bound for $ \phi $ can now be obtained by using the Green operator as in \cite{Alesker-Shelukhin (2013)}. We give here some details for convenience of the reader. By a quaternionic version of Gauduchon theorem \cite[Proposition 2.2]{Alesker-Shelukhin (2013)},
        there exists a pointwise strictly positive q-real $(2n,0)$-form $\Theta$ (which might not be holomorphic) such that $ \partial \partial_J \left(\Omega^{n-1}_0\wedge\bar{\Theta}\right)=0. $ In addition, we may  normalize $\Theta$ so that $ \int_{M}\Omega^{n}_0\wedge\bar{\Theta}=1.$ By \cite[Lemma 23]{Alesker-Shelukhin (2013)}, the quaternionic Laplacian admits a non-negative Green function $G(p, q)\geq 0$, namely, for each function $u$ of class $C^{2}$ and each point $p\in M$, 
            $$
            -\int_{q\in M}G(p,q)\Delta_gu(q)\,\Omega^{n}_0\wedge \bar{\Theta}=u(p)-\int_{M}u\,\Omega^{n}_0\wedge \bar{\Theta}\,.
            $$
        Choose a point $p\in M$ such that $\phi$ attains its maximum at $p$. Since we assumed $\sup_M\phi=0$ we have
        \begin{equation*}
             \|\phi\|_{L^{1}}= \int_{M}(-\phi)\,\Omega^{n}_0\wedge\bar{\Theta} = -\int_{q\in M}G(p,q)\Delta_g\phi(q)\,\Omega^{n}_0\wedge\bar{\Theta} \leq C\int_{q\in M}G(p,q)\,\Omega^{n}_0\wedge\bar{\Theta}\leq C\,.
        \end{equation*}
		
		Alternatively an $ L^p $-bound can be obtained by using the weak Harnack inequality as follows. Take an open cover of $ M $ made of coordinate balls $ B_{2r_i}(x_i) $ such that $ \{ B_i=B_{r_i}(x_i) \} $ still covers $ M $. Since $ \phi $ is non-positive and it satisfies the elliptic inequality \eqref{eq_in_C0}, the weak Harnack inequality \cite[Theorem 9.22]{GT} implies 
		\[
		\|\phi\|_{L^p(B_i)}=\left( \int_{B_i}(-\phi)^p \right)^{1/p}\leq C\left( \inf_{B_i}(-\phi)+1 \right)
		\]
		where $ p,C>0 $ depend only on the cover and the background metric. Since $ \sup_M \phi=0 $ there is at least one index $ j $ such that $ \inf_{B_j}(-\phi)=-\sup_{B_j}\phi=0 $, and thus $ \|\phi\|_{L^p(B_j)}\leq C $. This bound can be extended to all balls $ B_i $ such that $ B_i\cap B_j\neq \emptyset $, indeed the estimate on $ \|\phi\|_{L^p(B_j)} $ yields an upper bound for $ \inf_{B_i}(-\phi) $ as
		\[
		\inf_{B_i} (-\phi)\leq \inf_{B_i\cap B_j}(-\phi) \leq \frac{1}{\mathrm{Vol}(B_i\cap B_j)^{1/p}}\|\phi\|_{L^p(B_i\cap B_j)}\leq \frac{1}{\mathrm{Vol}(B_i\cap B_j)^{1/p}}\|\phi\|_{L^p(B_j)}\,.
		\]
		We can now reiterate the argument and in a finite number of steps we will have bound $ \|\phi\|_{L^p(B_i)} $ for each $ i $, and thus also $ \|\phi\|_{L^p(M)}. $
	\end{proof}

	\begin{prop}\label{prop_C0}
		Let $ (M,I,J,K,g) $ be a compact locally flat hyperhermitian manifold. If $ \underline \phi,\phi $ are a $ \mathcal{C} $-subsolution and a solution to \eqref{eq_main} respectively, with $ \sup_M \phi=0 $, then there is a constant $ C>0 $, depending only on the background data and the subsolution $ \underline{\phi} $, such that
		\[
		\|\phi\|_{C^0}\leq C\,.
		\]
	\end{prop}
	\begin{proof}
		Without loss of generality we may assume that $ \underline{\phi}\equiv 0 $, otherwise we could modify $ \Omega $ to simplify the equation. Since $ \sup_M \phi=0 $, we only need to bound $ S=\inf_M \phi $ from below. For convenience, we may assume $S\leq -1$, otherwise we are done.
		
		Since $ \underline{\phi} $ is a $ \mathcal{C} $-subsolution there exist $ \delta,R>0 $ such that
		\begin{equation}
			\label{eq_in_C0_2}
			\left( \lambda\left( g^{\bar j r}\Omega_{\bar j s}\right)-\delta {\bf 1}+\Gamma_n \right)\cap \partial  \Gamma^{h(x)}\subseteq B_R(0)\,, \qquad \text{at every }x\in M\,,
		\end{equation}
		where $ {\bf 1}=(1,1,\dots,1) $.
		
		Consider quaternionic local coordinates $ (q^1,\dots,q^{n}) $ centered at the point where $ \phi $ attains its minimum $ S $. We may identify such coordinate neighborhood with the open ball of unit radius $ B_1=B_1(0)\subseteq \H^{n} $ centered at the origin. Let $ v(x)=\phi(x)+\epsilon|x|^2 $ be defined on $ B_1 $ for some small fixed $ \epsilon>0 $. Observe that $ \inf_{B_1} v=v(0)=\phi(0)=S $ and $ \inf_{\partial B_1}v\geq v(0)+\epsilon $. These conditions allow us to apply the ABP method (see \cite[Proposition 10]{Sze}) to obtain
		\begin{equation}
			\label{eq_in_C0_3}
			C_0\epsilon^{4n}\leq \int_P \det(D^2v)\,,
		\end{equation}
		where $ C_0>0 $ is a dimensional constant,
		\[
		P=\left\{ x\in B_1\mid |Dv(x)|<\frac{\epsilon}{2},\, v(y)\geq v(x)+Dv(x)\cdot (y-x) \text{ for all }y\in B_1 \right\}\,,
		\]
		and $ Dv $, $ D^2v $ are the gradient and the (real) Hessian of $ v $. Note that $ P\subseteq \{ x\in B_1 \mid D^2v(x)\geq 0  \} $ and since convexity implies quaternionic plurisubharmonicity (see e.g. \cite{Alesker (2003)}), at any point $ x\in P $ we have $ \mathrm{Hess}_\H v(x)\geq 0 $. Therefore $ \mathrm{Hess}_\H \phi(x)\geq -\frac{\epsilon}{2} \mathbbm{1} $, where $ \mathbbm{1} $ is the $n \times n $ identity matrix. Choosing $ \epsilon $ small enough depending on $ g $ and $ \delta $, we have
		\[
		\lambda\left( g^{\bar j r}(\Omega_{\bar j s}+\phi_{\bar j s})\right)\in \lambda \left( g^{\bar j r}\Omega_{\bar j s} \right)-\delta {\bf 1}+\Gamma_n\,, \qquad \text{at every }x\in P\,.
		\]
		On the other hand, equation \eqref{eq_main} also gives
		\[
		\lambda\left( g^{\bar j r}(\Omega_{\bar j s}+\phi_{\bar j s})\right)\in \partial \Gamma^{h(x)}\,, \qquad \text{at every }x\in P\,.
		\]
		These two facts, together with \eqref{eq_in_C0_2} imply $ |\phi_{\bar r s}|\leq C $ on $ P $ and thus also $ v_{\bar r s} \leq C $. Combining a calculation in \cite{Blocki} with \cite[Lemma 2]{Sro}, or alternatively using directly a computation in the proof of \cite[Proposition 2.1]{Alesker-Shelukhin (2013)}, at any point $ x\in P $ we have
		\[
		\det(D^2v) \leq 2^{4n}\det(\mathrm{Hess}_\H(v))^4\,,
		\]
		where, on the right-hand side, ``$ \det $'' denotes the Moore determinant, introduced in \cite{Moore} (see also e.g. \cite{Alesker (2003),Aslaksen,SThesis}). Therefore, from \eqref{eq_in_C0_3} we see that
		\[
		C_0\epsilon^{4n}\leq C \mathrm{Vol}(P)\,.
		\]
		The definition of $ P $ entails that $ v(0)\geq v(x)-Dv(x)\cdot x>v(x)-\epsilon/2 $, i.e. $ v(x)<S+\epsilon/2<0 $ for all $ x\in P $. As a consequence for any $ p>0 $
		\[
		\|v\|^{p}_{L^p(M)}\geq \|v\|^{p}_{L^p(P)}=\int_P (-v)^p \geq 
		%\left(\int_P \left \lvert S+\frac{\epsilon}{2}\right \rvert^p\right)^{1/p}=
		\left \lvert S+\frac{\epsilon}{2}\right \rvert^p\mathrm{Vol}(P)\,.
		\]
		From the previous lemma we know that there is a $ p>0 $ such that $ \|v\|_{L^p} $ is bounded, therefore also $ S=\inf_M \phi $ must be bounded.
	\end{proof}
	
	\section{Laplacian estimate}\label{Laplacian}
	This section is devoted to derive a $ C^0 $-estimate for the quaternionic Laplacian of solutions to \eqref{eq_main} in terms of the squared norm of the gradient. This step is the most involved in terms of calculations and it is here that we use our strongest assumptions to have a locally flat hypercomplex structure and a hyperk\"ahler metric compatible with it. 
	
	We follow Székelyhidi \cite{Sze} and Hou-Ma-Wu \cite{Hou-Ma-Wu}, which in turn is based on an idea of Chou and Wang \cite{Chou-Wang} for the real Hessian equation. Our restrictive assumptions simplify quite a bit the computations.
	\medskip
	
	As declared in the introduction, let $ F(A)=f(\lambda(A)) $ be a symmetric function of the eigenvalues of $ A_{rs}=g^{\bar j r}\Omega^\phi_{\bar j s}=g^{\bar j r}(\Omega_{\bar j s}+\phi_{\bar j s}) $. We denote the derivatives of $ F $ by
	\[F^{rs}=\frac{\partial F}{\partial A_{rs}},\qquad F^{rs,lt}=\frac{\partial^{2} F}{\partial A_{rs}\partial A_{lt}}.\]
	Let $Q_{rs}$ be the standard quaternionic coordinates on $\H^{n,n}$ and let $E_{rs}^p$ be the real coordinates underlying $Q_{rs}$, i.e. $Q_{rs}=E_{rs}^0+E_{rs}^1i+E_{rs}^2j+E_{rs}^3k$. We have the following:
	%{\color{red} Let $Q^{rs}$ be the standard quaternionic coordinates on $\H^{n,n}$ and let $E^{rs}_p$ be the real coordinates underlying $Q^{rs}$, i.e. $Q^{rs}=E^{rs}_0+E^{rs}_1i+E^{rs}_2j+E^{rs}_3k$}. We denote the derivatives of $ F $ by
	%\[ \color{red}
	%F^{rs}:=\partial_{Q^{rs}} F\,, \qquad  %F^{rs,lt}:=\partial_{ Q^{lt}} F^{rs}\,.
	%\]

	\begin{lem}
	The linearization of $F$ at $\phi$ is the operator
	\[
	    L(\psi)=\Re \, \sum_{r,s=1}^{n}F^{rs}g^{\bar j r} \psi_{\bar j s}\,.
	\]
	\end{lem}
	\begin{proof}
	With respect to the real coordinates $E_{rs}^p$ we decompose a matrix $A\in \H^{n,n}$ as $A^{rs}_pE_{rs}^p$. Define the derivatives $F^{rs}_p:=\frac{\partial F}{\partial A^{rs}_p}$ and the matrix $H=(F^{rs})$. For a curve of hyperhermitian matrices $A_t$ with respect to $g$ we have
	\[
	\frac{d}{dt} F(A_t)=\sum_{r,s=}^n \sum_{p=0}^3 F^{rs}_p(A_t)(A'_t)^{rs}_p=\Re\, F^{rs}(A_t)(A'_t)_{rs}
	%=\Re\, \mathrm{tr} \left( ^tH(A_t) A'_t\right)\,,
	\]
	%where $\cdot ^t$ denotes transposition. 
	Now, for each $\psi\in C^{2}(M,\R)$ and $t\in (-\epsilon, \epsilon)$, let $\varphi(t) $ be a curve of $\Gamma$-admissible functions in $C^{2}(M,\R)$ with $\varphi(0)=\phi$ and $\phi'(0)=\psi$ and set $A_t=g^{-1} \left( \Omega+ \mathrm{Hess}_\H \phi(t) \right)$, then
	\[
        L(\psi)=\frac{d}{dt}F(A_t) \big \vert_{t=0}= \Re\, F^{rs}(A_0) (A'_0)_{rs}=\Re \sum_{r,s=1}^n F^{rs}( A_0) g^{\bar j r}\psi_{\bar j s}\,. \qedhere
	\]
	\end{proof}

	In order to prove the desired bound we will need the following preliminary lemma. 
	
	\begin{lem}\label{lem_pre_C2}
		Let $ \sup_{\partial \Gamma}f<a<b<\sup_{\Gamma} f $ and $ \delta, R>0 $. Then there exists a constant $ \kappa>0 $ such that for any $ \sigma \in [a,b] $, $ B\in \mathrm{Hyp}(n, \H) $ satisfying
		\[
		\left( \lambda(B)-2\delta {\bf 1}+\Gamma_n \right)\cap \partial \Gamma^\sigma \subseteq B_R(0)\,,
		\]
		$A\in \mathrm{Hyp}(n,\H) $ satisfying $ \lambda(A)\in \partial \Gamma^\sigma $ and $ |\lambda(A)|>R $, we have 
		\begin{align*}
			&\mbox{either }&&  \Re\, F^{rs}(A)\left( B_{rs}-A_{rs} \right)>\kappa \sum_{r=1}^n F^{rr}(A)\,,\\
			&\mbox{or }&& F^{ss}(A)>\kappa \sum_{r=1}^n F^{rr}(A)\,, \qquad \text{for all } s=1,\dots,n\,.
		\end{align*}
	\end{lem}
	\begin{proof}
		The lemma follows from the very same argument as \cite[Proposition 6]{Sze} once we have proved a quaternionic analogue of the Schur-Horn theorem.
	\end{proof}
	
	\begin{lem}[Quaternionic Schur-Horn Theorem]
		Let $ \mu=(\mu_1,\dots,\mu_n),\lambda=(\lambda_1,\dots,\lambda_n) \in \R^n $ be such that $ \mu_1\geq \cdots \geq \mu_n $ and $ \lambda_1\geq \cdots \geq \lambda_n $. There exists a hyperhermitian matrix $ B $ with diagonal $ \mu $ and eigenvalues $ \lambda $ if and only if
		\begin{equation}\label{Schur-Horn}
			\sum_{i=1}^j \mu_i\leq \sum_{i=1}^j\lambda_i\,, \qquad \text{for all }j=1,\dots,n\, \qquad \text{and}\qquad \sum_{i=1}^n\mu_i=\sum_{i=1}^n \lambda_i\,.
		\end{equation}
	\end{lem}
	\begin{proof}
		A hyperhermitian matrix $ B $ satisfies the assumptions of the lemma if and only if there exists $ C\in \mathrm{Sp}(n) $ such that $ B=C^*DC $ where $ D $ is the diagonal matrix with diagonal $ \lambda $. In particular $ \mu $ is the diagonal of $ B $ if and only if $ \mu=T\lambda $ where $ T=(|c_{rs}|^2) $. Since $  C\in \mathrm{Sp}(n) $, the matrix $ T $ is doubly stochastic. By the Birkhoff theorem \cite{Birkhoff} $\mu=T\lambda$, where $T$ is doubly stochastic, if and only if $T$ lies in the convex hull of the set of all permutation matrices. In other words $ B $ exists if and only if $ \mu $ lies in the convex hull of the vectors obtained by permuting the entries of $ \lambda $, which is known to be equivalent to \eqref{Schur-Horn} (see e.g. \cite[Theorem 46]{HLP}).
	\end{proof}

	\begin{prop} \label{prop_C2-bound}
		Let $ (M,I,J,K,g) $ be a compact flat hyperk\"ahler manifold. If $ \underline \phi,\phi $ are a $ \mathcal{C} $-subsolution and a solution to \eqref{eq_main} respectively, then there is a constant $ C>0 $, depending only on $ (M,I,J,K) $, $ \|g\|_{C^2} $, $ \|h\|_{C^2} $, $ \|\Omega \|_{C^2} $, $ \|\phi \|_{C^0} $ and $ \underline{\phi} $, such that
		\[
		\|\Delta_g \phi \|_{C^0}\leq C\left( \|\nabla \phi \|_{C^0}^2+1 \right)\,.
		\]
Here $\nabla$ denotes the Obata connection on $M$. 
	\end{prop}
Let us remark that as pointed out by Alesker \cite[pp. 204]{Alesker (2013)}, $M$ admitting a flat hyperk\"ahler metric $g$ compatible with the hypercomplex structure implies that $g$ is  parallel with respect to the Obata connection, therefore the Obata connection and the Levi-Civita connection coincide.

	We observe that at a point where $A$ is diagonal with distinct eigenvalues we have 
	\begin{itemize}
		\itemsep0.2em
		\item $ \lambda_{i}^{rs}:=\frac{\partial \lambda_i}{\partial A_{rs}}=\delta_{ir}\delta_{is} $,
		\item $ \lambda_{i}^{rs,tl}:= \frac{\partial^2 \lambda_i}{\partial A_{rs}\partial A_{tl}}=(1-\delta_{ir})\frac{\delta_{is}\delta_{it}\delta_{rl}}{\lambda_i-\lambda_r}+(1-\delta_{it})\frac{\delta_{il}\delta_{ir}\delta_{st}}{\lambda_i-\lambda_t} $
	\end{itemize}
	(see e.g. \cite{Gerhardt,Spruck}).
	%\[
	%\lambda_{1,rs,tl}A_{rs,x^a_p}A_{tl,x^a_p}=\sum_{r>1}\frac{A_{r1,x^a_p}A_{1r,x^a_p}+A_{1r,x^a_p}A_{r1,x^a_p}}{\lambda_1-\lambda_r}
	%\]
	Furthermore, since $ F(A)=f(\lambda(A)) $ for $ f $ symmetric, then $ F^{rs}=\delta_{rs}f_r $, and
	%\item $ F^{rs,tl}=f_{rt}\delta_{rs}\delta_{tl}+\frac{f_r-f_s}{\lambda_r-\lambda_s}(1-\delta_{rs})\delta_{rl}\delta_{st} $,
	since $ f $ is  concave and satisfies $ f_i>0 $ (assumption C1 in the introduction), then $ F $ is concave and $ \frac{f_r-f_s}{\lambda_r-\lambda_s}\leq 0 $. In particular $ f_r\geq f_s $ anytime $ \lambda_r\leq \lambda_s $. Finally, we  observe that by \cite[Lemma 9 (b)]{Sze} for any fixed $ x\in M $  there is a constant $ \tau>0 $ depending on $ h(x) $ such that
	\begin{equation}\label{F}
		\sum_{a=1}^n F^{aa}(x)>\tau>0\,.
	\end{equation}
	We will mainly be interested in the largest eigenvalue $ \lambda_1 $ of the matrix $ A $ around some fixed point $ x_0 $. As pointed out by Sz\'{e}kelyhidi \cite{Sze} in order for $ \lambda_1\colon M\to \R $ to define a smooth function at $ x_0 $ we need the eigenvalues to be distinct; to be sure of that, we perturb the matrix $ A $.
	
	At any fixed point $x_0\in M$ we can perturb $A$ in order to have a matrix with distinct eigenvalues. Indeed, fix quaternionic local coordinates around the point $ x_0 $ such that, at $ x_0 $, $ A $ is diagonal and its eigenvalues satisfy
	\begin{equation}\label{eq_eigenvalues}
		\lambda_1\geq \lambda_2\geq \dots \geq \lambda_n\,;
	\end{equation}
	take a constant diagonal matrix $ D $ whose entries satisfy
	\[
	0=D_{11}<D_{22}<\dots<D_{nn}\,.
	\]
	The matrix $ \tilde A=A-D $ has, at $ x_0 $, the eigenvalues
	\[
	\tilde \lambda_1=\lambda_1\,, \qquad \tilde
	\lambda_i=\lambda_i-D_{ii}\,, \text{ for } i=2,\dots,n\,,
	\]
	which are distinct by construction.
	%and thus define smooth functions at the chosen point.
	
	We will make use of the linearized operator $ L $ defined by $ L(u)=4\Re \sum_{a,b=1}^nF^{ab}g^{\bar c a}u_{\bar c b} $, where $ u_{\bar c b}=\frac{1}{4}\partial_{\bar q^c}\partial_{q^b}u $.
	%To overcome issues with quaternionic derivatives we will mainly use real derivatives with respect to the underlying real coordinates $ (x^1_0,x^1_1,x^1_2,x^1_3,x^2_0,\dots,x^n_3) $. \medskip
	First of all, we prove the following inequality for $ L\left(2\sqrt{\tilde \lambda_1}\right) $.

	\begin{lem}\label{lem_pre_C2_2}
		With respect to quaternionic local coordinates around $x_0$ such that $ (g_{\bar rs}) $ is the identity at $x_0$ and  $(\Omega^{\phi}_{\bar rs})$ is diagonal at $x_0$, we have 
		\[
		L\left(2\sqrt{\tilde \lambda_1}\right)\geq -\frac{F^{aa}|\Omega^\phi_{\bar 1 1,a}|^2}{2\lambda_1\sqrt{\lambda_1}}-\frac{C\mathcal{F}}{\sqrt{\lambda_1}}\,,
		\]
		where $ \mathcal{F}=\sum_{a=1}^nF^{aa}(x_0) $, $ \Omega^\phi_{\bar 1 1,a}=\partial_{q^a}\Omega^\phi_{\bar 1 1} $ and $ C>0 $ is a positive constant depending only on $ (M,I,J,K) $, $ \|\Omega\|_{C^2} $ and $ \|h\|_{C^2} $.
	\end{lem}
	\begin{proof}
		We have for the perturbed matrix $ \tilde A_{rs}=A_{rs}-D_{rr}\delta_{rs}=g^{\bar j r}\Omega^\phi_{\bar j s}-D_{rr}\delta_{rs} $ at the point $ x_0 $ where $ (g_{\bar rs}) $ is the identity and $ A $ (and thus $ (F^{rs}) $) is diagonal
		\begin{equation}\label{eq_pre_C2_2}
			L\left(2\sqrt{\tilde \lambda_1}\right)=8\Re\, F^{ab}\left(\sqrt{\tilde \lambda_1}\right)_{\bar ab}=2F^{aa}\sum_{p=0}^3\left(\sqrt{\tilde \lambda_1}\right)_{x^a_px^a_p}=F^{aa}\sum_{p=0}^3\left( \frac{\tilde \lambda_{1,x^a_px^a_p}}{\sqrt{\lambda_1}}-\frac{\tilde \lambda_{1,x^a_p}^2}{2\lambda_1\sqrt{\lambda_1}} \right)\,,
		\end{equation}
		where the subscript $ x^a_p $ denotes the real derivative with respect to the corresponding coordinate. Using the formulas for the derivatives of the eigenvalues we obtain
		%\[
		%\tilde
		%\lambda_{1,x^a_p}=\tilde \lambda_{1,rs} \tilde A_{rs,x^a_p}=\Re\left(g^{\bar j 1}\Omega^\phi_{\bar j 1 x^a_p}-B_{11,x^a_p}\right)
		%\]
		%\[
		%\begin{split}
		%\tilde
		%\lambda_{1,x^a_px^a_p}&=\tilde \lambda_{1,rs,lt} \tilde A_{rs,x^a_p}\tilde A_{lt,x^a_p}+\tilde
		%\lambda_{1,rs}\tilde A_{rs,x^a_px^a_p}\\
		%&=\Re \left(\sum_{r>1}\frac{\tilde A_{r1,x^a_p}\tilde A_{1r,x^a_p}+\tilde A_{1r,x^a_p}\tilde A_{r1,x^a_p}}{\tilde \lambda_1-\tilde \lambda_r}+g^{\bar j 1}\Omega^\phi_{\bar j 1 x^a_px^a_p}-B_{11,x^a_px^a_p}\right)
		%\end{split}
		%\]
		at $ x_0 $
		\begin{align*}
			\tilde
			\lambda_{1,x^a_p}&=\tilde
			\lambda_1^{rs}\tilde A_{rs,x^a_p}=\Omega^\phi_{\bar 1 1, x^a_p}\\
			%\begin{align*}
			%\tilde A_{r1,x^a_p}\tilde
			%A_{1r,x^a_p}=&(\Omega^\phi_{\bar r 1x^a_p}-B_{r1,x^a_p})(\Omega^\phi_{\bar 1 rx^a_p}-B_{1r,x^a_p})\\
			%=&|\Omega^\phi_{\bar r1x^a_p}|^2-B_{r1,x^a_p}\Omega^\phi_{\bar 1r,x^a_p}-\Omega^\phi_{\bar r1,x^a_p}B_{1r,x^a_p}+B_{r1x^a_p}B_{1r,x^a_p}
			%\end{align*}
			%
			%\[
			%\tilde A_{1r,x^a_p}\tilde
			%A_{r1,x^a_p}=|\Omega^\phi_{\bar r1x^a_p}|^2-B_{1r,x^a_p}\Omega^\phi_{\bar r1,x^a_p}-\Omega^\phi_{\bar 1r,x^a_p}B_{r1,x^a_p}+B_{1r,x^a_p}B_{r1,x^a_p}
			%\]
			\tilde
			\lambda_{1,x^a_px^a_p}&=\tilde \lambda_{1}^{rs,lt} \tilde A_{rs,x^a_p}\tilde A_{lt,x^a_p}+\tilde
			\lambda_{1}^{rs}\tilde A_{rs,x^a_px^a_p}=\sum_{r>1}\frac{\tilde A_{r1,x^a_p}\tilde A_{1r,x^a_p}+\tilde A_{1r,x^a_p}\tilde A_{r1,x^a_p}}{\lambda_1-\tilde \lambda_r}+\Omega^\phi_{\bar 1 1, x^a_px^a_p}\\
			&=\sum_{r>1}\frac{A_{r1,x^a_p}A_{1r,x^a_p}+A_{1r,x^a_p}A_{r1,x^a_p}}{\lambda_1-\tilde \lambda_r}+g^{\bar j 1}\Omega^\phi_{\bar j 1, x^a_px^a_p}= 2 \sum_{r>1}\frac{|\Omega^\phi_{\bar r1,x^a_p}|^2}{\lambda_1-\tilde \lambda_r}+\Omega^\phi_{\bar 1 1, x^a_px^a_p}\,,
			%-2\Re \sum_{r>1}\frac{B_{r1,x^a_p}\Omega^\phi_{\bar 1r,x^a_p}+\Omega^\phi_{\bar r1,x^a_p}B_{1r,x^a_p}}{\lambda_1-\tilde \lambda_r}+\tilde \lambda_{1}^{rs,tl}B_{rs,x^a_p}B_{tl,x^a_p}\,.
		\end{align*}
		where we used that the derivatives of $ D $ vanish because it is a constant matrix.
		
		Differentiating the equation $ F(A)=h $ twice with respect to $ x^1_p $ gives, at the point $ x_0 $,
		%\[
		%F^{rr}_p \Omega^\phi_{\bar r r,x^1_p}=h_{x^1_p}\,,
		%\]
		\begin{equation}\label{eq_pre_C2}
			\Re\, F^{rs,tl}\Omega^\phi_{\bar sr,x^1_p}\Omega^\phi_{\bar lt,x^1_p}+F^{rr} \Omega^\phi_{\bar r r,x^1_px^1_p}=h_{x^1_px^1_p}\,.
		\end{equation}
		We observe that
		\[
		\sum_{p=0}^3\Omega^\phi_{\bar 1 1, x^a_px^a_p}=\sum_{p=0}^3\left(\Omega_{\bar 1 1, x^a_px^a_p}+\phi_{\bar 1 1 x^a_px^a_p}\right)=4\Omega_{\bar 1 1, \bar a a}+4\phi_{\bar a a \bar 1 1}=4\Omega_{\bar 1 1, \bar a a}-4\Omega_{\bar a a, \bar 1 1}+\sum_{p=0}^3\Omega^\phi_{\bar a a, x^1_p x^1_p}
		\]
		and thus, by \eqref{eq_pre_C2} and \eqref{F} 
		\[
		F^{aa}\sum_{p=0}^3\tilde
		\lambda_{1,x^a_px^a_p}\geq 
		%\frac{2}{n}\sum_{r>1}\frac{|\Omega^\phi_{\bar r1,x^a_p}|^2
			%-\frac14|\Omega^\phi_{\bar 1r,x^a_p}|^2-\frac14|\Omega^\phi_{\bar r1,x^a_p}|^2
			%}{\lambda_1}+\Omega^\phi_{\bar 1 1 x^a_px^a_p}-1\geq 
		F^{aa}\sum_{p=0}^3\Omega^\phi_{\bar 1 1, x^a_px^a_p}
		%-1
		%\geq F^{aa}\sum_{p=0}^3\Omega^\phi_{\bar a a, x^1_p x^1_p}-C\mathcal{F}
		\geq - \Re\,F^{rs,tl}\sum_{p=0}^3\Omega^\phi_{\bar r s,x^1_p}\Omega^\phi_{\bar t l,x^1_p}-C\mathcal{F}
		\geq -C\mathcal{F}
		\]
		where we also used the concavity of $ F $. Finally from \eqref{eq_pre_C2_2} we have the desired inequality
		\[
		L\left(2\sqrt{\tilde \lambda_1}\right)\geq -\frac{F^{aa}\sum_{p=0}^3(\Omega^\phi_{\bar 1 1,x^a_p})^2}{2\lambda_1\sqrt{\lambda_1}}-\frac{C\mathcal{F}}{\sqrt{\lambda_1}}\,. \qedhere
		\]
	\end{proof}

	\begin{proof}[Proof of Proposition $\ref{prop_C2-bound}$]
		We have already seen that the Laplacian is bounded from below, as a consequence of \eqref{cone}, therefore it is enough to obtain a bound of the form
		\[
		\frac{\lambda_1}{\|\nabla \phi\|^2_{C^0}+1}\leq C\,.
		\]
		
		Define the function
		\[
		G=2\sqrt{\tilde \lambda_1}+\alpha(|\nabla \phi|^2)+\beta(\phi)\,,
		\]
		where 
		\begin{align*}
			\alpha(t)&=-\frac{1}{2}\log \left( 1- \frac{t}{2N} \right)\,,& N&= \|\nabla \phi\|^2_{C^0}+1\,,\\
			\beta(t)&=-2St+\frac{1}{2}t^2\,,& S&>\|\phi\|_{C^0}\,, \text{ large constant to be chosen later}\,,
		\end{align*}
		and $ \tilde \lambda_1 $ is, as before, the highest eigenvalue of the perturbed matrix $ \tilde A $ around a point $ x_0 $, which we choose to be a maximum point of $ G $. The derivative of the functions $ \alpha $ and $ \beta $ satisfy
		\begin{align}\label{eq_alpha}
			\frac{1}{4N}<&\alpha'(|\nabla \phi|^2)<\frac{1}{2N}\,, & \alpha''&=2(\alpha')^2\,,\\
			\label{eq_beta}
			S\leq& -\beta'(\phi)\leq 3S\,, & \beta''&=1\,.
		\end{align}
		
		At $ x_0 $ we have $ L(G)\leq 0 $. Choose quaternionic local coordinates such that $(g_{\bar r s})$ is the identity in the whole neighborhood of $x_0$  and $ (\Omega^{\phi}_{\bar rs}) $ is diagonal at $x_0$.  This 
		is possible because we are assuming $g$  hyperk\"ahler and flat. Then
		\begin{equation}\label{eq_in_C2}
			0\geq 4\Re\, F^{ab}G_{\bar ab}=4F^{aa}G_{\bar a a}=F^{aa}\sum_{p=0}^3G_{x^a_p x^a_p}\,,
		\end{equation}
		because $ F^{ab} $ is diagonal at $ x_0 $. We compute the derivatives of $ G $ at $ x_0 $:
		\begin{align*}
			0=G_{x^a_p}=&\left(2\sqrt{\tilde
				\lambda_1}\right)_{x^a_p}+\alpha'\sum_{r=1}^n(\phi_{\bar r x^a_p}\phi_{r}+\phi_{\bar r}\phi_{rx^a_p})+\beta'\phi_{x^a_p}\,,\\
			G_{x^a_px^a_p}=&\left(2\sqrt{\tilde \lambda_1}\right)_{x^a_px^a_p}+\alpha''\left(\sum_{r=1}^n(\phi_{\bar r x^a_p}\phi_{r}+\phi_{\bar r}\phi_{rx^a_p})\right)^2\\
			&+\alpha'\sum_{r=1}^n(\phi_{\bar rx^a_px^a_p}\phi_{r}+2|\phi_{r x^a_p}|^2+\phi_{\bar r}\phi_{rx^a_px^a_p})+\beta''\phi_{x^a_p}^2+\beta'\phi_{x^a_px^a_p}\,.
		\end{align*}
		Differentiating the equation $ F(A)=h $ yields
		\[
		F^{aa}\Omega^\phi_{\bar aa,x^r_p}=h_{x^r_p}\,, \qquad \text{at }x_0.
		\]
		Using this, Cauchy-Schwarz inequality and \eqref{eq_alpha} we have
		\[
		\begin{split}
			\alpha'F^{aa}\sum_{r=1}^n(\phi_{\bar r \bar a a}\phi_r+\phi_{\bar r}\phi_{r\bar a a})&=\alpha'F^{aa}\sum_{r=1}^n(\phi_{\bar aa\bar r}\phi_r+\phi_{\bar r}\phi_{\bar a ar})\\
			&=\alpha'\sum_{r=1}^n\left((h_{\bar r}-F^{aa}\Omega_{\bar aa, \bar r})\phi_r+\phi_{\bar r}(h_{r}-F^{aa}\Omega_{\bar aa, r})\right)\\
			&\geq -\frac{C}{N}(N^{1/2}+N^{1/2}\mathcal{F})\geq -C\mathcal{F}\,,
		\end{split}
		\]
		where we used \eqref{F} to absorb the constants into $ C\mathcal{F} $. Again using \eqref{eq_alpha} we also obtain
		\[
		\begin{split}
			2\alpha'F^{aa}\sum_{r=1}^n\sum_{p=0}^3|\phi_{r x^a_p}|^2&\geq  \frac{1}{2N}F^{aa}\sum_{r=1}^n\sum_{p,q=0}^3\phi_{x^r_q x^a_p}^2\geq \frac{1}{2N}F^{aa}\sum_{p=0}^3\phi_{x^a_p x^a_p}^2=\frac{8}{N}F^{aa}\phi_{\bar a a}^2\\
			&=\frac{8}{N}F^{aa}(\lambda_a-\Omega_{\bar aa})^2\geq \frac{2}{N}F^{aa}\lambda_a^2-C\mathcal{F}\,,
		\end{split}
		\]
		where, for the last inequality we used that $ (a+b)^2\geq \frac{1}{2}a^2-b^2 $. 
		Thanks to the last two inequalities, from our main inequality \eqref{eq_in_C2} we get
		\begin{equation}\label{eq_main_C2}
			0\geq  L\left(2\sqrt{\tilde \lambda_1}\right)+\alpha''F^{aa}\sum_{p=0}^3\left(2\sum_{r=1}^n\Re(\phi_{\bar r x^a_p}\phi_{r})\right)^2+\beta''F^{aa}|\phi_{a}|^2+4\beta'F^{aa}\phi_{\bar aa}+\frac{2F^{aa}\lambda_a^2}{N}-C\mathcal{F}\,.
		\end{equation}
		By $ G_{x^a_p}(x_0)=0 $ we have
		\begin{equation}\label{eq_in_C2_5}
			\begin{split}
				\alpha'' F^{aa}\left(2\sum_{r=1}^n\Re(\phi_{\bar r x^a_p}\phi_{r})\right)^2&=2 F^{aa}\left( \frac{\Omega^\phi_{\bar 1 1,x^a_p}}{\sqrt{\lambda_1}}+\beta'\phi_{x^a_p}\right)^2\\
				&\geq 2\epsilon \frac{F^{aa}(\Omega^\phi_{\bar 1 1,x^a_p})^2}{\lambda_1}-\frac{2\epsilon }{1-\epsilon}(\beta')^2F^{aa}\phi_{x^a_p}^2\,,
			\end{split}
		\end{equation}
		where we used the inequality $ (a+b)^2\geq \epsilon a^2-\frac{\epsilon}{1-\epsilon}b^2\,, $ which holds for $ \epsilon\in (0,1) $. Summing \eqref{eq_in_C2_5} over $ p $ and combining it with Lemma \ref{lem_pre_C2_2} we obtain from \eqref{eq_main_C2}
		\begin{equation}\label{eq_in_C2_10}
			\begin{split}
				0\geq & \left(4\epsilon\sqrt{\lambda_1}-1\right)\frac{F^{aa}|\Omega^\phi_{\bar 1 1,a}|^2}{2\lambda_1\sqrt{\lambda_1}}+\left(\beta''-\frac{2\epsilon (\beta')^2}{1-\epsilon}\right)F^{aa}|\phi_{a}|^2+4 \beta'F^{aa}\phi_{\bar aa}+\frac{2F^{aa}\lambda_a^2}{N}-C\mathcal{F}\,.
			\end{split}
		\end{equation}
		Choosing $ \epsilon=1/(18S^2+1)<1 $, \eqref{eq_beta} implies
		\[
		\beta''-\frac{2\epsilon}{1-\epsilon}(\beta')^2\geq 0\,. 
		\]
		Furthermore, we can assume without loss of generality  $ \sqrt{\lambda_1}>\frac{1}{4\epsilon} $ and deduce 
		\[
		\left(4\epsilon\sqrt{\lambda_1}-1\right)\frac{F^{aa}|\Omega^\phi_{\bar 1 1,a}|^2}{2\lambda_1\sqrt{\lambda_1}} \geq 0\,.
		\]
		Then we obtain from \eqref{eq_in_C2_10}
		\begin{equation}\label{eq_in_C2_7}
			0\geq 4\beta'F^{aa}\phi_{\bar aa}+\frac{2F^{aa}\lambda_a^2}{N}-C\mathcal{F}\,.
		\end{equation}
		
		As before, we can assume $ \underline{\phi}\equiv 0 $, otherwise we could choose a suitable background form $ \Omega $ in order to simplify the equation. Set $ B_{rs}=g^{\bar j r}\Omega_{\bar j s} $ and let $ \delta,R>0 $ be such that
		\[
		\left( \lambda(B)-2\delta {\bf 1}+\Gamma_n \right)\cap \partial \Gamma^{h(x)}\subseteq B_R(0)\,, \qquad \text{at every }x\in M\,,
		\]
		which exist because of the definition of $ \mathcal{C} $-subsolution. Supposing $ \lambda_1>R $ we have $ |\lambda(A)|>R $ and we can then apply Lemma \ref{lem_pre_C2} according to which there exists $ \kappa>0 $ such that one of the following two cases occur:
		\begin{itemize}
			\item
			First case: 
			$$
		 \Re\,F^{rs}(A)(B_{rs}-A_{rs})=- \Re\sum_{r,s=1}^nF^{rs}(A)g^{\bar j r}\phi_{\bar j s}>\kappa \sum_{r=1}^n F^{rr}(A)\,,
			$$
			i.e. $ -F^{aa}\phi_{\bar a a}>\kappa \mathcal{F} $ at $ x_0 $, which for a choice of $ S $ large enough implies $ 4\beta'F^{aa}\phi_{\bar a a}-C\mathcal{F}\geq 0 $ allowing us to deduce from \eqref{eq_in_C2_7} $ 0\geq \frac{2}{N}F^{aa}\lambda_a^2 $ which is a contradiction.
			
			\vspace{0.1cm}
			\item Second case:
			\begin{equation*}
			F^{ss}(A)>\kappa \sum_{r=1}^n F^{rr}(A)\,, \qquad \text{for all } s=1,\dots,n\,,
			\end{equation*}
			and in particular $ F^{11}>\kappa \mathcal{F} $. Therefore $ F^{aa}\lambda_a^2\geq F^{11}\lambda_1^2 \geq \kappa \mathcal{F}\lambda_1^2 $. Moreover, we can assume $ F^{aa}\lambda_a\leq F^{aa}\lambda_a^2/(12NS) $ for otherwise we would have $ \kappa \mathcal{F}\lambda_1^2<12NS\mathcal{F}\lambda_1 $ and we would conclude. Then we have
			\[
			4\beta' F^{aa}\phi_{\bar aa}\geq -12SF^{aa}\lambda_a-C\mathcal{F}\geq-\frac{F^{aa}\lambda_a^2}{N} -C\mathcal{F}\,.
			\]
			Substituting this last inequality into \eqref{eq_in_C2_7} we get
			\[
			0\geq \kappa\frac{\lambda_1^2}{N^2}-C\,.
			\]
		\end{itemize}
		This gives the bound we were searching for at the maximum point $ x_0 $ of $ G $, but by monotony of the square root such bound holds globally, depending additionally on a bound for $ \|\phi\|_{C^0} $.
	\end{proof}
	
	\begin{oss}
		Removing the hypothesis that the metric $ g $ is hyperk\"ahler one has to deal with its derivatives. Most of the terms are not an issue and can be easily controlled, however those terms that contain the third derivative of $ \phi $ seem not to be straightforwardly manageable.
	\end{oss}
	
	\begin{oss}
		The function $ G $ used in the proof of Proposition \ref{prop_C2-bound} is basically the same as the one used in \cite{Sze}, however we replaced the logarithm with the square root, a trick which is inspired by the work of Alesker \cite{Alesker (2013)}. It seems that using the square root allows to simplify the argument.
	\end{oss}
	
	\begin{oss}
	Under an additional assumption the Laplacian can be controlled linearly by the gradient. Indeed, if we further assume
\begin{equation}\label{assump}
    F^{aa}\lambda_{a}\leq c_{0}\,,
\end{equation}
which is the case for the quaternionic Monge-Amp\`ere, the quaternionic Hessian, and the quaternionic Monge-Amp\`ere equation for $(n-1)$-quaternionic plurisubharmonic functions, we obtain the following sharper estimate in the second case above, more precisely, from \eqref{eq_in_C2_7}, $ F^{11}>\kappa \mathcal{F} $ and \eqref{assump} we get
	\begin{equation*}\label{eq_in_C2_8}
	\begin{split}
	    	0\geq &\, 4\beta'F^{aa}(\lambda_{a}-1)+\frac{2F^{11}\lambda_1^2}{N}-C\mathcal{F}\geq 4\beta'F^{aa}\lambda_{a}+\frac{2\kappa\lambda_1^2}{N}\mathcal{F}+\left(-4\beta'-C\right)\mathcal{F}\\
	    	\geq &\, 4\beta'c_{0}+\frac{2\kappa\lambda_1^2}{N}\mathcal{F}+\left(-4\beta'-C\right)\mathcal{F}\geq  \frac{2\kappa\lambda_1^2}{N}\mathcal{F}+\left(-4\beta'-C+\frac{4\beta'c_{0}}{\tau}\right)\mathcal{F},
	\end{split}
		\end{equation*}		
where we have used $\mathcal{F}\geq \tau>0$ in the last inequality. Then we have
	\[
		0\geq 2\kappa\frac{\lambda_1^2}{N}-\left(4\beta'+C-\frac{4\beta'c_{0}}{\tau}\right)\,,
	\]
which gives a sharper bound 
	\[
    	\lambda_{1}\leq C(1+\|\nabla \varphi\|_{C^{0}})\,.
	\]
\end{oss}

	\section{Blow-up analysis}\label{Blow-up}
	In this section we show that a bound for the gradient of solutions to \eqref{eq_main} can be obtained by using a Liouville-type theorem. We adapt the approach of Dinew and Ko\l{}odziej \cite{DK} to our setting.
	\medskip
	
	We introduce the following:
	\begin{defn}
		A continuous function $ u \colon \H^n \to \R $ is a \emph{(viscosity) $ \Gamma $-subsolution (resp. supersolution)} if for all $ \psi\colon \H^n \to \R $ of class $ C^2 $ such that $ u-\psi $ has a local maximum (resp. minimum) at $ p $, we have $ \lambda({\rm Hess}_\H \psi )\in \bar \Gamma $ (resp. $ \lambda({\rm Hess}_\H \psi )\in \R^n\setminus \Gamma $) at $ p $. We say that $ u $ is a \emph{(viscosity) $ \Gamma $-solution} if it is both a subsolution and a supersolution.
	\end{defn}
	
	We show that if the gradient bound for solutions to \eqref{eq_main} does not hold, we are able to find a bounded $ C^{1,\alpha} $ viscosity $ \Gamma $-solution $ u\colon \H^n \to \R $ with bounded gradient and such that $ |\nabla u(0)|=1 $. In particular $ u $ is non-constant. In the next section we prove a Liouville-type theorem for this kind of functions, thus yielding a contradiction and showing implicitly that the gradient bound holds.

	\medskip 
	
	Let $ (M,I,J,K,g) $ be a compact hyperhermitian manifold. Consider a sequence $ (\underline \phi_j)_j $, $ (\phi_j)_j $, $ (h_j)_j $ of real smooth functions
	on $ M $ and  a sequence $(\Omega_j)_j $  of q-real $ (2,0) $-forms on $M$  such that $\underline \phi_j $ are $ \mathcal{C} $-subsolutions and $ \phi_j $, $ h_j $, $ \Omega_j $ satisfy
	\begin{equation}\label{eq_in_blowup2}
		\begin{cases}
			F\left(g^{\bar t r}((\Omega_j)_{\bar t s}+(\phi_j)_{\bar t s})\right)=h_j\,,\\
			\sup_M \phi_j=0\,,\\
			\|\nabla \phi_j\|_{C^0} \geq j\,.
		\end{cases}
	\end{equation}
	Assume further that $ (\underline{\phi}_j)_j $, $ (h_j)_j $ and $ (\Omega_j)_j $ are uniformly bounded in $ C^2 $-norm.

	Set $ N_j=\|\nabla \phi_j\|_{C^0}^2 $, $ g_j=N_jg $ and let $ x_j\in M $ be such that $|\nabla \phi_j(x_j)|^2=N_j $ for each $ j>0 $. Choose quaternionic local coordinates $(q^1,\dots,q^{n})$ around $ x_j $ for $ |q^i|<N^{1/2}_j $ such that
	\begin{align*}
		(g_j)_{\bar rs}=\delta_{\bar rs}+ O(N^{-1}_j|x|)\,, && (\Omega_j)_{\bar rs}=O(N^{-1}_j)\,, && h_j=h_j(x_j)+ O(N^{-1}_j|x|)\,.
	\end{align*}
	Then $ |\nabla \phi_j(x_j)|^2_{g_j}=1 $ and by Propositions \ref{prop_C0} and \ref{prop_C2-bound} we have in this coordinates
	\[
	\|\phi_j\|_{C^0}\leq C\,, \qquad 
	|\Delta_g \phi_j|_{g_j} \leq C\,, \qquad \text{on }B_{N^{1/2}_j}(x_j)\,,
	\]
	where $ C>0 $ is uniform in $ j $. It follows by \cite[Theorem 8.32]{GT} that $ (\phi_j)_j $ is uniformly bounded in $ C^{1,\alpha} $-norm for any $ \alpha \in (0,1) $. Furhermore, letting $ j\to \infty $, we see that $ \Omega_j $ tends to zero, while $ g_j $ tends to the standard Euclidean metric and $ (\phi_j)_{\bar r s} $ stays bounded. Therefore
	\begin{equation}
		\label{eq_in_blowup}
		\lambda(A_j)=\lambda((\phi_j)_{\bar r s})+O(N^{-1}_j|x|)\,,
	\end{equation}
	where $ (A_j)^r_s=g^{\bar tr}_j((\Omega_j)_{\bar ts}+(\phi_j)_{\bar ts}) $.
	
	By Ascoli-Arzel\`{a} Theorem we can extract from $ (\phi_j)_j $ a subsequence converging uniformly in $ C^{1,\alpha} $ to some $ u\colon \H^n\to \R $, moreover, such limiting function satisfies $ \|u\|_{C^0}\leq C $, $ \|\nabla u\|_{C^0}\leq C $ and $ |\nabla u (0)|=1 $. We aim to prove that $ u $ is a $ \Gamma $-solution.
	
	Suppose there exists $ \psi\in C^2 $, such that $ u-\psi $ has a local maximum at some point $ p_0\in \H^n $. By construction of $ u $, for any $ \epsilon>0 $ there are a $ j $ large enough, $ a\in (-\epsilon,\epsilon) $ and a point $ p_1 $ with $ |p_1-p_0|<\epsilon $ such that $ \phi_j-\psi-\epsilon|x-p_0|^2+a $ has a local maximum at $ p_1 $. As a consequence the quaternionic Hessian of $ \psi $ satisfies
	\[
	\mathrm{Hess}_\H\psi+\frac{\epsilon}{2} {\mathbbm 1}\geq \mathrm{Hess}_\H\phi_j\,, \qquad \text{at } p_1\,,
	\]
	where $\mathbbm{1}$ is the $n\times n$ identity matrix. By \eqref{eq_in_blowup}, if $ j $ is large enough we see that $ \lambda(\mathrm{Hess}_\H\psi )\in \Gamma-\epsilon {\bf 1} $ at $ p_1 $ and letting $ \epsilon \to 0 $ we deduce $ \lambda(\mathrm{Hess}_\H \psi)\in \bar \Gamma $ at $ p_0 $ because $ p_1 \to p_0 $. This shows that $ u $ is a viscosity $ \Gamma $-subsolution.
	
	To see that $ u $ is also a $ \Gamma $-supersolution we proceed similarly. Suppose that $ u-\psi $ has a local minimum at $ p_0\in \H^n $, then for any $ \epsilon>0 $ there are $ j $ large enough, $ a\in (-\epsilon,\epsilon) $ and $ p_1\in \H^n $ such that $ \phi_j-\psi+\epsilon|x-p_0|^2+a $ has a local minimum at $ p_1 $. Hence
	\[
	\mathrm{Hess}_\H \psi-\frac{\epsilon}{2} {\mathbbm 1}\leq \mathrm{Hess}_\H\phi_j\,, \qquad \text{at } p_1\,.
	\]
	By contradiction, suppose $ \lambda(\mathrm{Hess}_\H\psi(p_1)) \in \Gamma+\frac{5}{2}\epsilon{\bf 1} $, then $ \lambda(\mathrm{Hess}_\H\phi_j(p_1))\in \Gamma+2\epsilon{\bf 1} $ and for $ j $ large enough \eqref{eq_in_blowup} we have $ \lambda(A_j)\in \Gamma+\epsilon{\bf 1} $. By \cite[Lemma 9 (a)]{Sze} it follows that for $ N_j $ large enough $ \Gamma+N_j\epsilon{\bf 1}\subseteq \Gamma^{h_j(p_1)} $ and consequently we deduce  
	\[
	N_j\lambda(A_j)\in N_j\Gamma +N_j\epsilon{\bf 1}=\Gamma +N_j\epsilon{\bf 1}\subseteq \Gamma^{h_j(p_1)}
	\]
	for  $ j $ sufficiently large. On the other hand, $ \phi_j $ satisfies \eqref{eq_in_blowup2}, i.e.
	\[
	N_j\lambda(A_j)=\lambda\left( g^{\bar tr }((\Omega_j)_{\bar ts}+(\phi_j)_{\bar ts}) \right)\in \partial \Gamma^{h_j(p_1)}\,,
	\]
	which gives a contradiction. Therefore $ \lambda(\mathrm{Hess}_\H \psi(p_1)) \notin \Gamma+\frac{5}{2}\epsilon{\bf 1} $ and letting $ \epsilon\to 0 $ we finally obtain $ \lambda(\mathrm{Hess}_\H\psi(p_0)) \notin \Gamma $ and $ u $ is a viscosity $ \Gamma $-solution.\\
	
	\section{Liouville-type theorem}\label{Liouville}
	As in Székelyhidi \cite{Sze} we can interpret the notion of being a $ \Gamma $-subsolution (resp. solution) as that of being a viscosity subsolution (resp. solution) of a suitable equation. Indeed, define the function $ G_0 $ on the space of hyperhermitian matrices as the function such that
	\[
	\lambda(A)-G_0(A){\bf 1}\in \bar \Gamma,
	\]
	consider the projection $ {\rm p} \colon \R^{4n,4n} \to \{H\in \R^{4n,4n} \mid I_0HI_0=J_0HJ_0=K_0HK_0=-H\} $
	\[
	{\rm p}(H)=\frac{1}{4}(H-I_0HI_0-J_0HJ_0-K_0HK_0)\,,
	\]
	where $ (I_0,J_0,K_0) $ is the standard hyperhermitian structure on $ \R^{4n} $ written in block form as
	\begin{equation}\label{eq_I0J0K0}
		I_0=\begin{pmatrix}
			0 & -\mathbbm{1} & 0 & 0\\
			\mathbbm{1} & 0 & 0 & 0\\
			0 & 0 & 0 &-\mathbbm{1}\\
			0 & 0 & \mathbbm{1} & 0
		\end{pmatrix},\quad
		J_0=\begin{pmatrix}
			0 & 0 & -\mathbbm{1} & 0\\
			0 & 0 & 0 & \mathbbm{1}\\
			\mathbbm{1} & 0 & 0 & 0\\
			0 & -\mathbbm{1} & 0 & 0
		\end{pmatrix},\quad
		K_0=\begin{pmatrix}
			0 & 0 & 0 & -\mathbbm{1}\\
			0 & 0 & -\mathbbm{1} & 0\\
			0 & \mathbbm{1} & 0 & 0\\
			\mathbbm{1} & 0 & 0 & 0
		\end{pmatrix},
	\end{equation}
	where $\mathbbm{1}$ is the $n\times n$ identity matrix. Then, defining the function $ G $ on the space of $ 4n\times 4n $ symmetric matrices $ \mathrm{Sym}(4n,\R) $ as $ G(H)=G_0({\rm p}(H)) $, we have that $ u $ is a $ \Gamma $-subsolution (resp. solution) if and only if it is a viscosity subsolution (resp. solution) of the equation $ G(D^2u)=0 $.
	
	Therefore we can take advantage from the known results regarding viscosity subsolutions and solutions (see \cite{CC}). In particular we will use the following:
	\begin{itemize}
		\itemsep0.2em
		\item If $ (u_j)_j $ is a sequence of $ \Gamma $-subsolutions (resp. solutions) converging locally uniformly to $ u $, then $ u $ is a $ \Gamma $-subsolution (resp. solution) as well.
		\item If $ u,v $ are $ \Gamma $-subsolutions, then $ u+v $ is a $ \Gamma $-subsolution as well.
		\item A mollification of a $ \Gamma $-subsolution is again a $ \Gamma $-subsolution.
	\end{itemize}
	
	We will also need the following comparison result
	
	\begin{lem}\label{lem_comparison}
		If $ u $ is a $ \Gamma $-solution and $ v $ a smooth $ \Gamma $-subsolution on a bounded open set $ U \subseteq \H^n $ such that $ u=v $ on $ \partial U $, then $ u\geq v $ in $ U $.
	\end{lem}
	\begin{proof}
		The very same proof of \cite[Lemma 17]{Sze}, which is the analogous result in $ \C^n $, can be carried out in our hypothesis.
	\end{proof}
	
	The next lemma follows from the same argument as \cite[Lemmas 18-19]{Sze}. The additional case when $ \Gamma=\Gamma_n $ is quite easy and can be deduced along the same lines.
	\begin{lem}\label{lem_pre_Liouville2}
		Suppose $ v\colon \H^n \to \R $ is a $ \Gamma $-solution which is independent of the last variable $ q_n $. Define
		\begin{equation}
			\label{eq_cone}
			\Gamma'=\begin{cases}
				\Gamma_{n-1} & \text{if } \Gamma=\Gamma_n\,,\\
				\Gamma\cap \{ x_n=0 \} & \text{if } \Gamma \neq \Gamma_n\,,
			\end{cases}
		\end{equation}
		then $ \Gamma' $ is a symmetric proper convex open cone in $ \R^{n-1} $ containing $ \Gamma_{n-1} $ and the function $ w(q_1,\dots,q_{n-1})=v(q_1,\dots,q_{n-1},0) $ is a $ \Gamma' $-solution on $ \H^{n-1} $.
	\end{lem}
	
	We remark that in view of \eqref{cone} every $ \Gamma $-subsolution is subharmonic.
	
	\begin{prop}[Liouville-type Theorem] \label{prop_Liouville}
		A Lipschitz bounded viscosity $ \Gamma $-solution $ u\colon \H^n \to \R $ with $ \|\nabla u\|_{C^0}\leq C $ is constant.
	\end{prop}
	\begin{proof}
		The result is proved by induction over $ n $. For $ n=1 $ the function $ u $ is harmonic and the result is well-known.
		
		Assume now that the result holds for $ n-1 $ and let us prove it for $ n $. By contradiction we suppose that $ u $ is not constant and $ \inf_M u=0 $, $ \sup_M u=1 $. We adopt the notation of \cite{Sze} and, for any function $ v\colon \H^n \to \R $ we write its mollification
		\[
		[v]_r(q)=\int_{q'\in \H^{n}} v(q+rq')\psi(q')\, \mathrm{dV}\,,
		\]
		where, here and hereafter, $ \mathrm{dV} $ denotes the standard volume form in $\H^n$ and $ \psi\colon \H^{n}\to \R $ is a smooth mollifier with support in $ B_1(0) $ such that $ \psi>0 $ in $ B_1(0) $ and $ \int_{\H^{n}}\psi\, \mathrm{dV}=1 $. During the proof we will need to regularize $ u $, considering $ u^\epsilon=[u]_\epsilon $ for a small $ \epsilon>0 $. Following \cite{DK} we use Cartan's Lemma to deduce
		\[
		\lim_{r\to \infty}[u^2]_r(q)=\lim_{r\to \infty}[u]_r(q)=1\,.
		\]
		
		For $ \rho>0 $ and $ r>0 $ consider the set
		\[
		U(\rho,r)=\left\{ q\in \H^n \mid 2u(q)\leq [u^2]_r(q)+[u]_\rho(q)-\frac{4}{3}  \right\}\,.
		\]
		Suppose there are $ \rho>0 $, $ \epsilon_j\to 0 $, $ q_j\in \H^{n} $, $ r_j\to \infty $ and a unit vector $ \xi_j\in \H^{n} $ such that $ q_j\in U(\rho,r_j) $ and
		\begin{equation}\label{eq_in_Liouville2}
			\lim_{j\to \infty}\int_{B_{r_j}(q_j)}|\bar \partial_{\xi_j}u^{\epsilon_j}|^2\mathrm{dV}=0\,,
		\end{equation}
		where for any vector $ \xi=(\xi^1_0+\xi^1_1i+\xi^1_2j+\xi^1_3k, \dots,\xi^n_0+\xi^n_1i+\xi^n_2j+\xi^n_3k) \in \H^n $ and any function $ w\colon \H^n\to \R $ we use the notation
		\[
		\bar \partial_{\xi} w=\sum_{r=1}^n \left(\xi^r_0 w_{x^r_0}+\xi^r_1w_{x^r_1}i+\xi^r_2w_{x^r_2}j+\xi^r_3w_{x^r_3}k\right)\,.
		\]
		Composing with rotations and translations, for each $ j $ we can take $ q_j $ to the origin and assume $ \xi_j=q^{n}/2 $, obtaining a sequence $ (u_j)_j $ of $ \Gamma $-solutions satisfying
		\begin{equation}\label{eq_in_Liouville}
			[u^2_j]_{r_j}(0)+[u_j]_\rho(0)-2u_j(0)\geq \frac{4}{3}\,,\qquad \lim_{j\to \infty}\int_{B_{r_j}(0)}\left \lvert \bar \partial_{\frac{q^{n}}{2}}u^{\epsilon_j}_j\right \rvert^2\mathrm{dV}=0\,.
		\end{equation}
		Since $ u $ has bounded gradient, by the Ascoli-Arzelà Theorem, up to a subsequence, $ (u_j)_j $ converges locally uniformly to some $ v\colon \H^n \to \R $ which must be again a $ \Gamma $-solution with bounded gradient. Also $ u^{\epsilon_j}_j $ converges to $ v $ locally uniformly and working as in \cite{DK} we infer that $ v $ does not depend on the last variable $ q^n $.
		
		Indeed, if $ v $ were not constant along lines with fixed $ q'=(q^1,\dots,q^{n-1}) $, there would be $ a,b\in \H $ and a positive $ c\in \R $ such that $ v(q'_0,a)-v(q'_0,b)>2c $. Since the gradient of $ v $ is bounded from above, we could choose $ \delta $ small enough such that
		\[
		\inf\left\{ v(q',q^n) \mid |q'-q'_0|<\delta,\, |q^n-a|<\delta \right\}-\sup\left\{ v(q',q^n) \mid |q'-q'_0|<\delta,\, |q^n-b|<\delta \right\}>c\,.
		\]
		Let $ \xi \in \H^n $ be the unit vector with last entry $(b-a)/|b-a| $ and all others zero. Let $ \gamma $ be the segment joining $ (q',a'),(q',b')\in \H^n $, where $ b'-a'=b-a $, $ |q'-q'_0|<\delta $, $ |a'-a|<\delta$, $ |b'-b|<\delta $, then we would have
		\[
		\left \lvert \int_{\gamma} \bar \partial_\xi v\, d\xi  \right \rvert=\left \lvert v(q',b')-v(q',a') \right \rvert >c\,.
		\]
		Cauchy-Schwarz inequality would now give
		\[
		c^2<\left \lvert \int_{\gamma} \bar \partial_\xi v\, d\xi  \right \rvert^2\leq \left(\int_{\gamma} |\bar \partial_\xi v|^2 d\xi\right) \left(\int_\gamma d\xi\right)=|b-a|\int_{\gamma} |\bar \partial_\xi v|^2 d\xi \,.
		\]
		Let $ I_1,I_2,I_3 $ be intervals of length $ \delta $ all perpendicular to each other and to $ [a,b] $ in the $ q^n $-space. Using Fubini's theorem over the set $ B(q_0',\delta)\times [a,b]\times I_1\times I_2\times I_3 $ we would find a strictly positive lower bound for the integral of $ |\bar \partial_{q^n/2} v|^2 \mathrm{dV} $. But this would contradict the uniform convergence as the $ u_j $'s satisfy \eqref{eq_in_Liouville}. Therefore $ v $ does not depend on the last variable.
		
		The function $ w(q^1,\dots,q^{n-1})=v(q^1,\dots,q^{n-1},0) $ is then a $ \Gamma' $-solution, thanks to Lemma \ref{lem_pre_Liouville2}, where $ \Gamma' $ is the cone defined in \eqref{eq_cone}. By the induction hypothesis $ w $ is constant and then so is $ v $. But by Cartan's Lemma this contradicts the first of \eqref{eq_in_Liouville} because
		\[
		\frac{4}{3}\leq \lim_{j\to \infty }\left( [u^2_j]_{r_j}(0)+[u_j]_\rho(0)-2u_j(0) \right)=1+[v]_\rho(0)-2v(0)=1-v(0)\leq 1
		\]
		as $ v $ inherits from $ u $ the property that $ 0\leq v\leq 1 $.
		
		This means that \eqref{eq_in_Liouville2} cannot hold, in particular for all $ \rho>0 $, there exists $ c_\rho>0 $ such that if $ r>c_\rho $, for each $ q\in U(\rho,r) $, $ \epsilon<c^{-1}_\rho $ and unit vector $ \xi \in \H^n $ we must have
		\begin{equation}\label{eq_in_Liouville3}
			\int_{B_r(q)}|\partial_\xi u^\epsilon|^2 \mathrm{dV}>c_\rho\,.
		\end{equation}
		Define
		\[
		U'(\rho,r)=\left\{ q\in \H^n \mid 2u(q)< [u^2]_r(q)+[u]_\rho(q)-\frac{4}{3}  \right\}\subseteq U(\rho,r)\,.
		\]
		We may choose the origin so that $ u(0)<1/12 $, and $ \rho>0 $ and $ r>c $ big enough to have $ [u]_\rho(0)>3/4 $ and $ [u^2]_r(0)>3/4 $ which can be done by Cartan's Lemma. It follows that $ 0\in U'(\rho,r) $.
		
		Since $ \partial_{\bar q^i}\partial_{q^j} (u^\epsilon)^2=2u^\epsilon u^\epsilon_{\bar i j}+2u^\epsilon_{\bar i} u^\epsilon_j $, proceeding similarly as in \cite{Sze} we can use \eqref{eq_in_Liouville3} to prove that there exists a constant $ \delta>0 $ small enough to guarantee that $ [(u^\epsilon)^2]_r-\delta |q|^2 $ is a $ \Gamma $-subsolution over $ U'(\rho,r) $. By local uniform convergence also $ [u^2]_r-\delta |q|^2 $ is a $ \Gamma $-subsolution. Finally consider
		\[
		U''(\rho,r)=\left\{ q\in \H^n \mid 2u(q)< [u^2]_r(q)-\delta |q|^2+[u]_\rho(q)-\frac{4}{3}  \right\}\subseteq U'(\rho,r)
		\]
		and observe that since $ 0\leq u\leq 1 $ this set is bounded. The fact that $ u $ is a $ \Gamma $-solution and yet $ [u^2]_r(q)-\delta |q|^2+[u]_\rho(q)-\frac{4}{3} $ is a smooth $ \Gamma $-subsolution contradicts the comparison principle of Lemma \ref{lem_comparison}. We conclude that $ u $ must be constant.
	\end{proof}

	\section{Proof of Theorem \ref{teor_main}}\label{C2a}
	The main theorem follows once we obtain the $C^{2,\alpha}$-estimate. We obtain the desired bound in two ways, by using an analogue of Evans-Krylov theory as developed in Tosatti-Wang-Weinkove-Yang \cite{Tosatti et al.} and by adapting the argument of B\l{}ocki \cite{Blo05} similarly to what was done by Alesker \cite{Alesker (2013)} for the treatment of the quaternionic Monge-Amp\`ere equation.
	
	\begin{prop}\label{prop_C2a}
		Let $ (M,I,J,K,g) $ be a compact locally flat hyperhermitian manifold. If $ \phi $ is a solution to \eqref{eq_main} such that $ \|\phi\|_{C^0} $ and $ \Delta_g\phi $ are bounded from above, then there is $ \alpha \in (0,1) $ and a constant $ C>0 $, depending only on the background data such that
		\[
		\|\phi \|_{C^{2,\alpha}}\leq C\,.
		\]
	\end{prop}
	\begin{proof}
		Let $ V=\{H\in \R^{4n,4n} \mid I_0HI_0=J_0HJ_0=K_0HK_0=-H\} $, where $ (I_0,J_0,K_0) $ is the standard hypercomplex structure on $ \R^{4n} $ as in \eqref{eq_I0J0K0}. Consider the real representation of quaternionic matrices  
		$\iota \colon\mathbb H^{n,n}\to V $, defined as
		$$ 
		\iota(A+iB+jC+kD):=
		\begin{pmatrix}
			A  & B & C & D\\
			-B & A & -D &C\\
			-C & D & A &-B\\
			-D & -C & B& A
		\end{pmatrix}\,.
		$$
		The map $\iota$ is an isomorphism of real algebras and $ \iota({\rm Hyp}(n, \H)) = V\cap \mathrm{Sym}(4n,\R) $. Let ${\rm p}\colon \R^{4n,4n}\to V $ be the projection
		$$
		{\rm p}(H):=\frac14(H-I_0HI_0-J_0HJ_0-K_0HK_0)\,.
		$$
		If we take on $ \H^n $ the real coordinates $ (x^1_0,\dots,x^n_0,x^1_1,\dots,x^n_1,x^1_2,\dots,x^n_2,x^1_3,\dots,x^n_3) $ underlying the quaternionic coordinates $ (q^1,\dots,q^n) $, for a $ C^2 $ function $ u\colon \H^n \to \R $ we have 
		\[
		\iota({\rm Hess}_\mathbb{H}u)=16 {\rm p}(D^2u)\,.
		\]
		
		For any point $ x_0\in M $, take a quaternionic coordinate chart centered at $ x_0 $ and assume that the domain of the chart contains $ B_1(0) $. For any $ H\in \mathrm{Sym}(4n,\R) $ we have $ \iota^{-1}(p(H))\in \mathrm{Hyp}(n,\H) $, therefore
		\[
		\tilde H_{rs}(x)=g^{\bar j r}(x)(\iota^{-1}(p(H)))_{\bar j s}\,,\qquad x\in B_1(0)\,,
		\]
		is hyperhermitian with respect to $ g $.
		
		Define the set
		\[
		\mathcal{E}=\left\{ H \in \mathrm{Sym}(4n,\R) \mid \lambda(\tilde H(0)) \in \bar \Gamma^{\sigma}\cap \overline{B_{2R}(0)} \right\}\,,
		\]
		where $ \sigma $ and $ R $ are chosen below. $ \mathcal{E} $ is compact and also convex by convexity of $ \Gamma $. Possibly shrinking $ B_1(0) $ to a smaller radius $ r\in (0,1) $
		%By a simple covering argument this will not be an issue
		we may assume that if $ H $ lies in a sufficiently close neighborhood $ U $ of  $ \mathcal{E} $, then $ \lambda(\tilde H(x))\in \bar \Gamma^{\sigma}\cap \overline{B_{4R}(0)} $ for any $ x\in B_1(0) $.
		
		The bound $ \Delta_g \phi \leq C $ implies that $ \sigma $ and $ R $ can be chosen so that
		\[
		\lambda\left(g^{\bar j r} \left( \Omega_{\bar j s}+\phi_{\bar j s} \right) \right)\in \bar \Gamma^\sigma \cap \overline{B_R(0)}\,, \qquad \text{on }B_1(0) \,.
		\]
		Therefore, by continuity of $ g $, and possibly shrinking $ B_1(0) $ again, for each $ x\in B_1(0) $ we have
		\[
		\iota (\Omega_{\bar r s}(x))+16\mathrm{p}(D^2\phi(x))= \iota \left( \Omega_{\bar r s}(x)+\phi_{\bar r s}(x) \right) \in \mathcal{E}\,.
		\]
		
		This discussion and our assumptions on $ f $ show that we can apply \cite[Theorem 1.2]{Tosatti et al.} with
		\begin{itemize}
			\itemsep0.2em
			\item $ F\colon \mathrm{Sym}(4n,\R)\times B_1(0)\to \R $ defined as $ F(H,x)=f(\lambda(\tilde H(x))) $ for $ H\in U $, and extended smoothly to all of $ \mathrm{Sym}(4n,\R)\times B_1(0) $;
			\item $ S\colon B_1(0)\to \mathrm{Sym}(4n,\R) $ defined as $ S(x)=\iota(\Omega_{\bar r s} (x)) $;
			\item $ T\colon \mathrm{Sym}(4n,\R)\times B_1(0)\to \mathrm{Sym}(4n,\R) $ defined as $ T(H,x)=16\mathrm{p}(H) $.
		\end{itemize}
		And since $ \|\phi\|_{C^0}\leq C $ we obtain the desired bound $ \|\phi\|_{C^{2,\alpha}}\leq C $ for some $ \alpha \in (0,1) $.
	\end{proof}
	
	Now we present our second proof.
\begin{proof}
	Since $M$ is locally flat, we only need to prove the following interior $C^{2,\alpha}$ estimate for $w=\phi+u$, where $u\in C^{\infty}_{\rm loc}(M,\R)$ is a local potential for $\Omega$.

	Now, $w\in C^{4}(\mathcal{O})$ satisfies
	\begin{equation*}
		F(w_{\bar rs})=h\,,
	\end{equation*}
	where $\mathcal{O}\subset \mathbb{H}^{n}$ is an arbitrary open subset and $h\in C^{\infty}(\mathcal{O})$. Let $\mathcal{O}'\subset\mathcal{O}$ be a relatively compact open subset. We shall prove that there exist a constant $\alpha\in (0,1)$ depending only on $n$, $h$,  $\|w\|_{C^{0}(\mathcal{O})}$,  $\|\Delta w\|_{C^{0}(\mathcal{O})}$ and a constant $C$ depending in addition on $\text{dist}(\mathcal{O},\mathcal{O}')$ such that 
\begin{equation*}\label{}
	\|w\|_{C^{2,\alpha}(\mathcal{O})}\leq C\,.
\end{equation*}

There is a difference with respect to the argument of Alesker \cite{Alesker (2013)}: the quaternionic Monge-Amp\`ere operator can be written in the divergence form, while this might not be true for more general fully non-linear equations. To overcome this issue we will need a more general version of the weak Harnack inequality for second order uniformly elliptic operators.

%In what follows, we let $B_{R}\subset \mathbb{R}^{N}$ be a ball of radius $R$ and center $0$. 
%\begin{teor}\label{wh}
%	Let $u\in C^{2}(B_{2R},\R)$ satisfy 
%	\begin{equation*}
%		\begin{cases}
%		 u\geq 0 & \text{on } B_{2R}\,;  \\
%		\sum_{i,j=1}^Na_{ij}D_{i}D_{j}u\leq \psi& \text{on } B_{R}\,.
%		\end{cases}
%	\end{equation*}
%with $\psi\in L^{N}(\mathcal{O})$, $a_{ij}=a_{ji}\in L^{\infty}(B_{R})$, and there are positive constants $\lambda, \Lambda$ such that 
%\begin{equation}\label{elliptic}
%	\lambda |\xi|^{2}\leq \sum_{i,j=1}^Na_{ij}\xi_{i}\xi_{j}\leq \Lambda |\xi|^{2}\,.
%\end{equation}
%Then there exists a positive constant $C$  depending only on $N$, $\lambda$, $\Lambda$ and $R$ such that
%	\begin{equation*}
%		\frac{1}{|B_{R}|}\int_{B_{R}}u\leq C\left(\inf_{B_{R}}u+\frac{R}{\lambda}\|\psi\|_{L^{N}(B_{2R})}\right)\,.
%	\end{equation*}	
%\end{teor}

Let $W$ be the quaternionic Hessian $(w_{\bar r s})$ and define a second order linear operator $\mathcal{D}$ by 
\begin{equation*}
\mathcal{D}v=\Re\, F^{rs}(W)v_{\bar r s}\,.	
\end{equation*}
Notice that every $n\times n$  hyperhermitian matrix defines a hyperhermitian semilinear form on $\mathbb{H}^{n}$. Hence it also determines a symmetric bilinear form on $\mathbb{R}^{4n}$. Let $(a_{ij})\in \text{Sym} (4n,\R)$  be the realization of $(F^{rs}(W))$. Then we can rewrite $\mathcal{D}v$ in the following form
\begin{equation*}
	\mathcal{D}v=\sum_{r,s=1}^{4n}a_{rs}D_{r}D_{s}v\,,
\end{equation*}
Since $F$ is uniformly elliptic on $\Gamma$, 
%there exist positive constants  $\lambda, \Lambda$ depending on $\|h\|_{C^{0}(\mathcal{O})}$ and $\|\Delta w\|_{C^{0}(\mathcal{O})}$  such that \eqref{elliptic} holds. This also implies that 
the operator $\mathcal{D}$ is uniformly elliptic as well.

Let $R>0$ be such that the open ball $B_{2R}$ of radius $2R$ centered at a point $z_{0}\in \mathcal{O}'$ is contained in $\mathcal{O}$. For an arbitrary unitary vector $\xi\in \mathbb{H}^{n}$, we let $\Delta_{\xi}$ denote the Laplacian on any translate of the quaternionic line spanned by $\xi$. By  virtue of concavity of $F$, for any unitary vector $\xi\in \mathbb{H}^{n}$, we have 
	\begin{equation}\label{concavex}
	\Re\,	F^{rs}(W)\Delta_{\xi} (w_{\bar r s}) \geq \Delta_{\xi}h\,.
	\end{equation}
Consider the function 
$$
\hat{w}=\sup_{B_{2R}}\Delta_{\xi} w-\Delta_{\xi} w\,.
$$
it follows from \eqref{concavex} that $ \mathcal{D}\hat{w}\leq -\Delta_{\xi}h, $ where we used the fact $\Delta_{\xi} (w_{\bar rs})=(\Delta_{\xi} w)_{\bar rs}$.

Then, applying the weak Harnack inequality \cite[Theorem 9.22]{GT}, there exists a positive constant $C$ depending on $n$, $\|h\|_{C^{2}(\mathcal{O})}$ and $\|\Delta u\|_{C^{0}(\mathcal{O})}$ such that 
\begin{equation*}
	\frac{1}{\mathrm{Vol}(B_{R})}\int_{B_{R}}\hat{w}
	%\leq C\big(\inf_{B_{R}}\hat{w}+\frac{R}{\lambda}\|\Delta_{\xi}h\|_{L^{4n}(B_{2R})}\big)	
	\leq C\left(\inf_{B_{R}}\hat{w}+R\right).
\end{equation*}	
Equivalently, we have
\begin{equation}\label{weakh}
	\frac{1}{\mathrm{Vol}(B_{R})}\int_{B_{R}}	\left(\sup_{B_{2R}}\Delta_{\xi} w-\Delta_{\xi} w\right)\leq C\left(\sup_{B_{2R}}\Delta_{\xi} w-\sup_{B_{R}}\Delta_{\xi} w+R\right)\,.
\end{equation}
Since $F$ is concave on $\Gamma$ for any pair of  $A\,,B\in {\rm Hyp}(n, \H)$, we have
\begin{equation*}
	F(B)-F(A)\leq \Re\, F^{rs}(A)(B_{rs}-A_{rs})\,.
\end{equation*}
Choosing $A=W(y)$ and $B=W(x)$ for $x,y\in B_{2R}$, it follows that
\begin{equation}\label{concave}
\begin{split}
\Re\,	F^{rs}(W(y))(w_{\bar rs}(y)-w_{\bar r s}(x)) 
	\leq F(W(y))-F(W(x))
	=h(y)-h(x)
    \leq C\|y-x\|
\end{split}
\end{equation}
for some positive constant $C$ depending on $\|h\|_{C^{1}(\mathcal{O})}$.

Now we need the following lemma from matrix theory, which is well-known in the settings of  $\mathbb{R}^{n}$, $\mathbb{C}^{n}$, $\mathbb{H}^{n}$ (see e.g. \cite{GT,Blo00,Alesker (2013)}).

\begin{lem}\cite[Lemma 4.9]{Alesker (2013)}
Let $\lambda,\Lambda\in \mathbb{R}$ satisfy $0<\lambda<\Lambda<+\infty$. There exist a uniform constant $N$, unit vectors $\xi_{1},\cdots, \xi_{N}\in \mathbb{H}^{n}$ and positive numbers $\lambda_{*}< \Lambda_{*}<+\infty$, depending only on $n, \lambda, \Lambda$ such that any  $A \in {\rm Hyp}(n, \H)$ with eigenvalues lying in the interval $[\lambda, \Lambda]$ can be written in the form 
\begin{equation*}
	A=\sum_{k=1}^{N}\beta_{k}\xi_{k}^*\otimes\xi_{k}\,, \qquad i.e. \, A_{rs}=\sum_{k=1}^{N}\beta_{k}\bar \xi_{kr}\xi_{ks}\,,
\end{equation*}
for some $\beta_{k}\in [\lambda_{*}, \Lambda_{*}]$.
\end{lem}

We apply the previous lemma with $A=(F^{rs}(W))$, obtaining immediately
\begin{equation*}
	\begin{split}
\Re\,F^{rs}(W(y))(w_{\bar rs}(y)-w_{\bar rs}(x))
=&\sum_{k=1}^{N}\beta_{k}(y)\bar \xi_{kr} \xi_{ks}(w_{\bar r s}(y)-w_{\bar rs}(x))\\	=&\sum_{k=1}^{N}\beta_{k}(y)(\Delta_{\xi_{k}}w(y)-\Delta_{\xi_{k}}w(x))
	\end{split}
\end{equation*}
for some functions $\beta_{k}(y)\in [\lambda_{*}, \Lambda_{*}].$ By \eqref{concave}, we then have
\begin{equation}\label{concav}
\sum_{k=1}^{N}\beta_{k}(y)(\Delta_{\xi_{k}}w(y)-\Delta_{\xi_{k}}w(x))	\leq C\|y-x\| \qquad \text{for } x,y\in B_{2R}\,.
\end{equation}
Let us denote 
\begin{equation*}
M_{k,tR}=\sup_{B_{tR}}\Delta_{\xi_{k}}w\,,\qquad m_{k,tR}=\inf_{B_{tR}}\Delta_{\xi_{k}}w\,,\qquad	\eta(tR)=\sum_{k=1}^{N}(M_{k,tR}-m_{k,tR})\,,
\end{equation*}
for $t=1,2$. 

Summing up \eqref{weakh} over $\xi_{k}$ for $k\neq l$ yields
\begin{equation}\label{weakha}
	\frac{1}{\mathrm{Vol}(B_{R})}\int_{B_{R}}\sum_{k\neq l}\big(M_{k,2R}-\Delta_{\xi_{k}} w\big)\leq C(\eta(2R)-\eta(R)+R)\,.
\end{equation}
Choosing a point $x\in B_{2R}$ at which the infimum $m_{l,2R}$ is attained, by \eqref{concav} we also know that
\begin{equation}\label{weakhaer}
	\begin{split}
&\Delta_{\xi_{l}}w(y)-m_{l,2R}	\leq\frac{1}{\lambda_{*}} \left(CR+\Lambda_{*}\sum_{k\neq l}(M_{k,2R}-\Delta_{\xi_{k}}w)\right)
	\end{split}
\end{equation}
Integrating \eqref{weakhaer} on $B_{R}$ and using \eqref{weakha} yields
\begin{equation*}
		\frac{1}{\mathrm{Vol}(B_{R})}\int_{B_{R}}
		(\Delta_{\xi_{l}}w-m_{l,2R})	%\leq \frac{1}{\lambda_{*}} \Big(CR+\frac{\Lambda_{*}}{\mathrm{Vol}(B_{R})}\int_{B_{R}}\sum_{k\neq l}(M_{k,2R}-\Delta_{\xi_{k}}w)\Big)
		\leq  C(\eta(2R)-\eta(R)+R)\,.
\end{equation*}
Using \eqref{weakh} again, we then obtain
\begin{equation*}
	\begin{split}
		\frac{1}{\mathrm{Vol}(B_{R})}\int_{B_{R}}
		(\Delta_{\xi_{l}}w-m_{l,2R})
			\geq&\,\frac{1}{\mathrm{Vol}(B_{R})}\int_{B_{R}}
		(\Delta_{\xi_{l}}w-M_{l,2R})+M_{l,2R}-m_{l,2R}\\
		\geq & \, M_{l,2R}-m_{l,2R}-C(M_{l,2R}-M_{l,R}+R)\\
		\geq & \, C(M_{l,R}-m_{l,R})-(C-1)(M_{l,2R}-m_{l,2R})-CR\,,
	\end{split}
\end{equation*}
since $m_{k,tR}$ is non-increasing with respect to $t$. Inserting this last inequality into \eqref{weakha} we get 
\begin{equation*}
\eta(2R)-\eta(R)\geq 	 C(M_{l,R}-m_{l,R})-(C-1)(M_{l,2R}-m_{l,2R})-CR\,,
\end{equation*}
and summing up over $l$,
\begin{equation*}
\eta(R)\leq (1-1/C)\eta(2R)+CR.	
\end{equation*}
Now applying \cite[Lemma 8.23]{GT} the proof is complete.
\end{proof}

	\begin{proof}[Proof of Theorem $ \ref{teor_main} $]
		Let $ (M,I,J,K,g) $ be a compact flat hyperk\"ahler manifold, $ \underline{\phi},\phi \colon M \to \R $ be a $ \mathcal{C} $-subsolution and a solution to \eqref{eq_main} respectively, with $ \sup_M \phi=0 $. By Proposition \ref{prop_C0} we deduce $ \|\phi \|_{C^0} \leq C $. Proposition \ref{prop_C2-bound} now implies $ \| \Delta_g \phi \|_{C^0} \leq C(\| \nabla \phi \|_{C^0}^2+1) $. The blow-up argument together with the Liouville-type Theorem \ref{prop_Liouville} yield a gradient bound for $ \phi $. Therefore $ \|\Delta_g \phi \|_{C^0}\leq C $ and we can deduce from Proposition \ref{prop_C2a} the desired $ C^{2,\alpha} $-estimate $ \|\phi\|_{C^{2,\alpha}}\leq C $, where the constant $ C>0 $ only depends on the background data, including $ \underline{\phi} $.
	\end{proof}

	\section{Proof of Theorems \ref{teor_Hessian} and \ref{teor n-1 MA}}\label{Hessian}
	In this section we prove Theorem \ref{teor_Hessian} and Theorem \ref{teor n-1 MA} as applications of Theorem \ref{teor_main}. For the quaternionic Hessian equation as the cone $ \Gamma $ we consider the $ k $-positive cone
	\[
	\Gamma_k=\{ \lambda \in \R^n \mid \sigma_1(\lambda),\dots,\sigma_k(\lambda)>0 \}\,,
	\]
	where $ 1\leq k \leq n $ and $ \sigma_r $ is the $ r $-th elementary symmetric function
	$$
	\sigma_r(\lambda)=\sum_{1\leq i_1<\dots<i_r\leq n} \lambda_{i_1}\cdots \lambda_{i_r}\,, \qquad \text{ for all }\lambda=(\lambda_1,\dots,\lambda_n)\in \R^n\,.
	$$
	Observe that on a locally flat hyperhermitian manifold $ (M,I,J,K,g) $ a q-real $ (2,0) $-form $ \Omega $ is $ k $-positive in the sense that it satisfies \eqref{k-pos} if and only if $ \lambda(g^{\bar j r}\Omega_{\bar j s}) \in \Gamma_k $.
	
	Moreover, for every $ (\lambda_1,\dots,\lambda_n)\in \Gamma_k $ we clearly have
	\[
	\lim_{t\to \infty} \sigma_k(\lambda_1,\dots,\lambda_{n-1},t)=\infty
	\]
	and by \cite[Remark 8]{Sze} any $\Gamma$-admissible function is a $ \mathcal{C} $-subsolution. Hence for the quaternionic Hessian equation we easily have existence of a $ \mathcal{C} $-subsolution.
	%
	%The Hessian equation on manifolds has been first investigated by Li \cite{Li} and Urbas \cite{Urbas} in the Riemannian case (see also the survey of Wang \cite{Wang}). Later some partial results have been obtained in the K\"ahler setting by Hou \cite{Hou}, Jbilou \cite{Jbilou} and Kokarev \cite{Kokarev} independently. The solution in its full generality on compact K\"ahler manifolds came by Dinew and Ko\l{}odziej \cite{DK} building on the estimate of Hou, Ma and Wu \cite{Hou-Ma-Wu}. The equation has also been solved on compact Hermitian and almost Hermitian manifolds (see \cite{TW,CTW} for the case $ k=n $ and \cite{Zhang,CHZ} for the general case).
	
	%Tosatti and Weinkove \cite{TW} managed to prove Yau's result without the K\"ahler assumption and together with Chu \cite{CTW} they even solved the complex Monge-Amp\`{e}re equation in the almost Hermitian case. The solvability of the complex Hessian equation on Hermitian manifolds was proved by D. Zhang \cite{Zhang} and on almost Hermitian manifolds by Chu, Huang and J. Zhang \cite{CHZ}.

	\begin{proof}[Proof of Theorem $ \ref{teor_Hessian} $]
		%Let $ x_0\in M $ and take a neighborhood around $ x_0 $ such that $ \hat g(x_0) $ is the identity and $ U(x_0) $ is diagonal. In these holomorphic coordinates $ (z^1,\dots,z^{2n}) $ we have
		%\begin{align*}
		%\Omega=\sum_{i=1}^{n}dz^{2i-1}\wedge dz^{2i}\,, && \Omega_\phi=\sum_{i=1}^{n}\lambda_i dz^{2i-1}\wedge dz^{2i}\,,
		%\end{align*}
		%where $ \lambda=(\lambda_1,\dots,\lambda_n)\in \R^n $ are the eigenvalues of $ U(x_0) $. Therefore
		%\[
		%\begin{split}
		%\Omega_\phi^k\wedge \Omega^{n-k}&=\sum_{\{i_1,\dots,i_k,j_1,\dots,j_{n-k}\}=\{1,\dots,n\}} \lambda_{i_1}\dots \lambda_{i_k} dz^{2i_1-1}\wedge dz^{2i_1} \wedge \dots \wedge dz^{2j_{n-k}-1}\wedge dz^{2j_{n-k}}\\
		%&=(n-k)!\sum_{\{i_1,\dots,i_k\}=\{1,\dots,n\}} \lambda_{i_1}\dots \lambda_{i_k} dz^{1}\wedge \dots \wedge dz^{2n}\\
		%&=\frac{(n-k)!}{n!}\sum_{1\leq i_1<\dots<i_k\leq n } \lambda_{i_1}\dots \lambda_{i_k} \Omega^n\\
		%&=\frac{1}{\binom{n}{k}} \sigma_k(\lambda) \Omega^n\,.
		%\end{split}
		%\]
		%
		%In this perspective equation \eqref{qh} rewrites as
		%\[
		%\sigma_k(\Omega_\phi)=\mathrm{e}^F\,,\quad
		%\int_M  \varphi\,\Omega^n \wedge\bar \Omega^n=0
		%\]
		On $ \Gamma_k $ we define $ f=\log \sigma_k $, in order to rewrite the quaternionic Hessian equation as  
		\[
		f\left(\lambda\left(g^{\bar j r}(\Omega_{\bar j s}+\phi_{\bar j s})\right)\right)=h\,,
		\]
		for some positive $ h\in C^\infty(M,\R) $ depending on $ H $. The function $f$ satisfies conditions C1--C3 stated in the introduction (see e.g. \cite{Spruck}). 
		
		We apply the method of continuity. Let $ H_0\in C^\infty(M,\R) $ be the function such that
		\[
		\frac{\Omega^{k} \wedge \Omega^{n-k}_0}{\Omega^{n}_0}=\mathrm{e}^{H_0}
		\]
		and consider the $ t $-dependent family of equations
		\begin{equation}\tag{$ *_t $}\label{eq_continuity} 
				\frac{\Omega^{k}_{\phi_t} \wedge \Omega^{n-k}_0}{\Omega^{n}_0}=b_t\,\mathrm{e}^{tH+(1-t)H_0}, \quad \varphi_{t}\in {\rm QSH}_{k}(M,\Omega),  \quad t\in [0,1]\,.
		\end{equation} 
		Let
		\[
		S=\left\{ t\in [0,1] \mid \eqref{eq_continuity} \text{ has a solution }(\phi_t,b_t)\in C^{2,\beta}(M,\R)\times \R_+ \right\}\,.
		\]
		By our choice of $H_0 $, the pair $ (\phi,b)=(0,1) $ solves $ (*_0) $, hence the set $ S $ is non-empty.

		Since we assumed $ \Omega $ to be $ k $-positive $ \underline{\phi}\equiv 0 $ is  $\Gamma_k $-admissible and therefore a $ \mathcal{C} $-subsolution. Closedness of $ S $ now follows from the $ C^{2,\alpha} $-estimate of Theorem \ref{teor_main}, a standard bootstrapping argument and the Ascoli-Arzel\`a Theorem.
		
		Finally, in order to show that $S$ is open, take $ t'\in S $ and let $ (\phi_{t'},b_{t'}) $ be the corresponding solution to $ (*_{t'}) $. Consider the Banach spaces
		\[
		B_1:=\left\{ \psi \in C^{2,\beta}(M,\R) \mid \psi\in {\rm QSH}_k(M,\Omega)\,, \, \int_M \psi\, \Omega_0^n\wedge \bar \Omega_0^n=0 \right\}\,, \qquad B_2:=C^{0,\beta}(M,\R) \,,
		\]
		and the linearization of the operator
		\[
		B_1\times \R_+ \to B_2\,, \qquad (\psi,a) \mapsto \log \frac{\Omega^{k}_{\psi} \wedge \Omega^{n-k}_0}{\Omega^{n}_0} -\log(a)
		\]
		at $ (\phi_{t'},b_{t'}) $, which is
		\[
		L\colon T_{\phi_{t'}}B_1\times \R \to B_2\,, \qquad L(\rho,c)=
		%k\frac{\partial \partial_J \rho \wedge \Omega^{k-1}_{\phi_{t'}} \wedge \Omega^{n-k}_0}{\Omega^{k}_{\phi_{t'}}\wedge \Omega_0^{n-k}}-\frac{c}{b_{t'}}=
		k\frac{\partial \partial_J \rho \wedge \Omega^{k-1}_{\phi_{t'}} \wedge \Omega^{n-k}_0}{b_{t'}\,\mathrm{e}^{t'H+(1-t')H_0}\Omega_0^n}-\frac{c}{b_{t'}}=:L'(\rho)-\frac{c}{b_{t'}}\,,
		\]
		where
		\[
		T_{\phi_{t'}}B_1=\left\{ \rho \in C^{2,\beta}(M,\R) \mid \int_M \rho\, \Omega_0^n\wedge \bar \Omega_0^n=0 \right\}\,.
		\]
		By the maximum principle the kernel of the operator $ L' $ over $ C^{2,\beta}(M,\R) $ is the set of constant functions. Moreover the principal symbol of $L'$ is self-adjoint  and therefore $ L' $ has index zero, which implies that its formal adjoint $ (L')^* $ has one-dimensional kernel as well. 
		In order to show that $L$ is surjective, let $ \zeta\in C^{0,\beta}(M,\R)$ and choose $c\in \R$ such that $\zeta +c/ b_{t'} $ is orthogonal to $ \ker((L')^*) $. By the Fredholm alternative there exists $ \rho \in B_1 $ such that 
		$$
		L'(\rho)= \zeta +c/ b_{t'}
		$$
		and the surjectivity of $L$ follows. 
		
		By the inverse function theorem between Banach spaces
		% a neighborhood of $ t $ lies within $ S $, showing that
		$ S $ is open. This proves 
		%that $ S=[0,1] $ and in particular
		the existence of a solution to the quaternionic Hessian equation.
		
		Finally we show uniqueness. Suppose $ (\phi_1,b_1),(\phi_2,b_2) $ are both solutions and assume $ b_1\geq b_2 $; then
		\[
		\left(\Omega^k_{\phi_1}-\Omega^k_{\phi_2}\right)\wedge \Omega_0^{n-k}\geq 0\,,
		\]
		which can be rewritten as
		\[
		\partial \partial_J(\phi_1-\phi_2)\wedge \left(\sum_{i=0}^{k-1} \Omega^{k-i-1}_{\phi_1}\wedge \Omega^{i}_{\phi_2}\right)\wedge \Omega_0^{n-k}\geq 0\,.
		\]
		Since 
		$$
		\varphi 	\mapsto \frac{\partial \partial_J\phi\wedge \left(\sum_{i=0}^{k-1} \Omega^{k-i-1}_{\phi_1}\wedge \Omega^{i}_{\phi_2}\right)\wedge \Omega_0^{n-k}}{\Omega_0^n}
		$$
		is a second order linear elliptic operator without free term, by  the maximum principle we deduce $ \phi_1=\phi_2 $ and thus also $ b_1=b_2 $.
	\end{proof}
	
	%\begin{oss}
	%It is very likely that the $ C^0 $-estimate for the quaternionic Hessian equation on compact HKT manifolds can be achieved as in \cite{Sroka} by obtaining a Cherrier-type inequality. This argument is based on the proof of the same estimate for the complex Monge-Amp\`{e}re equation on compact Hermitian manifolds, which needs an inequality found by Cherrier \cite{Cherrier} and later exploited by Tosatti and Weinkove in \cite{TW} to deduce the required bound.
	%\end{oss}

\begin{proof}[Proof of Theorem $\ref{teor n-1 MA}$]
		Similarly as discussed in \cite{Sze}, let $T$ be the linear map given by
			\[
			T(\lambda) = \big(T(\lambda)_{1},\ldots,T(\lambda)_{n}\big)\,, \qquad
			T(\lambda)_{k} = \frac{1}{n-1}\sum_{i\neq k}\lambda_{i}\,,
			\]
			for every $\lambda\in\mathbb{R}^{n}$ and define
			\[
			f = \log\sigma_{n}(T), \qquad \Gamma = T^{-1}(\Gamma_{n})\,.
			\]
			It is straightforward to verify that the above setting satisfies the assumptions C1--C3 in the introduction. Let
			\[
			\Omega: =\Re \left(g^{\bar j s}(\Omega_{1})_{\bar j s}\right)\Omega_{0}-(n-1)\Omega_{1}\,.
			\]
			Thus, equation \eqref{n-1 MA} can be written as
			\[
			f(\lambda) = H+\log b\,, \qquad \lambda=\lambda\left(g^{\bar j r}(\Omega_{\bar j s}+\phi_{\bar j s})\right) \in \Gamma\,.
			\]
Then, Theorem \ref{teor n-1 MA} can be proved by a similar argument of Theorem \ref{teor_Hessian}, we give some details here.
	
We consider the following family of equations for $t\in[0,1]$:
	\begin{equation*}
	\begin{cases}
		\ \big(\Omega_{1}+\frac{1}{n-1}\big[(\Delta_{g} \varphi_t)\Omega_{0}-\partial\partial_J \varphi_t\big]\big)^{n} = {\rm e}^{tH+(1-t)H_{0}+c_{t}}\Omega_{0}^{n}\,, \\[2mm]
		\ \Omega_{1}+\frac{1}{n-1}\big[(\Delta_{g} \varphi_t)\Omega_{0}-\partial\partial_J \varphi_t\big] > 0\,, \quad \sup_{M}\varphi_t = 0\,,
	\end{cases}\eqno (*)_{t}
\end{equation*}
where $H_{0}=\log \frac{\Omega_{1}^{n}}{\Omega_{0}^{n}}$ and $c_{t}: [0,1]\rightarrow \mathbb{R}$ is a path from $c_{0}=0$ to $c_{1}=\log b$. Let us define
\[
S = \{t\in[0,1] \mid \text{ there exists a pair $(\varphi_{t},c_{t})\in C^{\infty}(M,\R)\times\mathbb{R}$ solving $(*)_{t}$ }\}\,.
\]
Note that $(\varphi_{0},c_{0})=(0,0)$ solves $(*)_{0}$ and hence $S\neq\emptyset$. To prove the existence of solutions to \eqref{n-1 MA equation 1},  it suffices to show that $S$ is both closed and open.

{\bf Step 1.} $S$ is closed.  We first show that $\{c_{t}\}$ is uniformly bounded. Suppose $\varphi_{t}$ achieves its maximum at the point $p_{t}\in M$, then the maximum principle yields that $\partial\partial_{J} \varphi_{t}$ is non-positive at $p_{t}$. Combining this with $(*)_{t}$, we obtain the upper bound for $c_{t}$:
\[
c_{t}\leq \left(-tH+H_{0}\right)(p_{t})
\leq C\,,
\]
for some $C$ depending only on $H$, $\Omega_{1}$ and $\Omega$. The lower bound of $c_{t}$ can be obtained similarly.

Observe that the positivity of $\Omega_{1}$ implies that $\underline{\varphi}\equiv 0$ is a $\mathcal{C}$-subsolution of $(*)_{t}$. Then $C^{\infty}$ a priori estimates of $\varphi_{t}$ follow from Theorem \ref{teor_main}. Combining this with the Arzel\`a-Ascoli theorem, we conclude that $S$ is closed.

{\bf Step 2.} $S$ is open.  Suppose there exists a pair $(\varphi_{\hat{t}},c_{\hat{t}})$ satisfies  $(*)_{\hat{t}}$. We shall prove that when $t$ is close to $\hat t$, there exists a pair $(\varphi_{t},c_{t})\in C^{\infty}(M,\R)\times\mathbb{R}$ solving $(*)_{t}$.  	

First of all, let $\Theta$ be a pointwise strictly positive $(2n,0)$-form with respect to $I$ which is $I$-holomorphic, namely $\bar{\partial}\Theta=0$. Equivalently, $\partial \bar \Theta=\partial_{J}\bar \Theta=0.$

For every function $\psi: M\rightarrow \mathbb{R}$ of class $C^{2}$, we define 
$$
L_{\hat{\varphi}}(\psi):=\frac{n}{n-1}\frac{\big((\Delta_{g} \psi)\Omega_{0}-\partial\partial_J \psi \big)\wedge \big(\Omega_{1}+\frac{1}{n-1}\big[(\Delta_{g} \hat{\varphi})\Omega_{0}-\partial\partial_J \hat{\varphi}\big]\big)^{n-1}}{\big(\Omega_{1}+\frac{1}{n-1}\big[(\Delta_{g} \hat{\varphi})\Omega_{0}-\partial\partial_J \hat{\varphi}\big]\big)^{n}}\,.
$$	
Since the operator $L_{\hat{\varphi}}$ is second order elliptic its symbol is self-adjoint, and therefore the index is zero. Then the classical maximum principle yields that
\begin{equation}\label{ker L1}
	\ker(L_{\hat{\varphi}})=\{\text{const}\}\,.
\end{equation}
Denote by $L_{\hat{\varphi}}^{*}$ the $L^{2}$-adjoint operator of $L_{\hat{\varphi}}$ with respect to the volume form
$$
\text{dvol}=\left(\Omega_{1}+\frac{1}{n-1}\big[(\Delta_{g} \hat{\varphi})\Omega_{0}-\partial\partial_J \hat{\varphi}\big]\right)^{n}\wedge\bar{\Theta}\,.
$$ 
By the index theorem, we know there is a non-negative function $\zeta$ such that
\begin{equation}\label{ker L*}
	\ker(L_{\hat{\varphi}}^{*})=\text{Span}\big\{\zeta\big\}\,.
\end{equation}
It follows from the strong maximum principle that $\zeta>0$. Up to a constant, we may and do assume 
$$\int_{M}\zeta\, \text{dvol}=1.$$
Define a Banach space
\[B_1: =\left\{\phi\in C^{2,\alpha}\mid \lambda\left(g^{\bar j r}(\Omega_{\bar j s}+\phi_{\bar j s})\right) \in \Gamma, \int_{M}\phi\zeta \, \text{dvol}=0\right\}.\]
It is easy to verify that the tangent space of $B_{1}$ at $\hat{\varphi}$
is given by
\[T_{\hat{\varphi}}B_{1} =\left\{\psi\in C^{2,\alpha}(M,\R)\mid \int_{M}\psi\zeta \, \text{dvol}=0\right\}\,.\]
Let us consider the map
\begin{equation*} 
	\tilde H (\phi,c)=\log\frac{\big(\Omega_{1}+\frac{1}{n-1}\big[(\Delta_{g} \varphi)\Omega_{0}-\partial\partial_J \varphi\big]\big)^{n}}{\Omega_{0}^{n}}-c\,,
\end{equation*}
which maps $B_{1}\times \mathbb{R}$ to $C^{0,\alpha}$. The linearized operator of $\tilde H$ at $(\hat{\varphi},\hat{t})$ is given by 
\begin{equation}\label{linear operator}
	L_{\hat{\varphi}}-c: T_{\hat{\varphi}}B_{1}\times \mathbb{R}\rightarrow  C^{0,\alpha}(M,\R)\,.
\end{equation}
On the one hand, for any real-valued $h\in  C^{0,\alpha}(M)$, there exists a unique real constant $c$ such that 
$$\int_{M}(h+c)\zeta \, \text{dvol}=0.$$ 
By \eqref{ker L*} and Fredholm theorem, there exists  a real function $\psi$ on $M$ such that
$ L_{\hat{\varphi}}(\psi)-c=h.$ 
Hence, the map $L_{\hat{\varphi}}-c$ is surjective. On the other hand, let $(\psi_{1}, c_{1})$ be a solution of $L_{\hat{\varphi}}(\psi)-c=0.$
By \eqref{ker L*} and Fredholm theorem again, we get $c_{1}=0$. Using \eqref{ker L1} and \eqref{linear operator}, we also obtain $\psi_{1}=0.$
Therefore, $L_{\hat{\varphi}}-c$ is injective.

As a consequence, we conclude that $L_{\hat{\varphi}}-c$ is bijective.
By the implicit function theorem,  we know that when $|t-\hat{t}|$ is small enough, there exists a pair $(\varphi_{t},c_{t})$ satisfying
\begin{equation*}
	\tilde H(\varphi_{t},c_{t})=tH+(1-t)H_0.
\end{equation*}

In the general case, when we assume $M$ is a compact manifold which admits a flat hyperk\"ahler metric $g$ compatible with the underlying hypercomplex structure, we may take $\Theta=\Omega^{n}$ and apply the previous procedure to show existence of solutions to \eqref{n-1 MA equation 1}.

Uniqueness can be obtained with a very similar technique as in Theorem \ref{teor_Hessian}, therefore we omit the proof here.
\end{proof}

Before we move on to the proof of Corollary \ref{cor} we need to lay down some preliminaries in linear algebra in order to mimic the proof of \cite[Corollary 1.3]{TW17}. Let $(M,I,J,K,g,\Omega_0)$ be a compact hyperhermitian manifold. Let $(z^1,\dots,z^{2n})$ be holomorphic coordinates with respect to $I$ and denote $\Lambda_I^{p,0}(M)$ the space of $(p,0)$-forms with respect to $I$. Consider the pointwise inner product $\langle \cdot,\cdot \rangle_{g}$ defined by
\[
\langle \alpha,\beta\rangle_g=\frac{1}{p!} g^{r_1 \bar s_1}\cdots g^{r_p \bar s_p} \alpha_{r_1 \cdots r_p} \overline{\beta_{s_1\cdots s_p}}\,, \qquad \text{for every } \alpha,\beta \in \Lambda^{p,0}_I(M)\,,
\]
where any $(p,0)$-form $\alpha $ is locally written as $\alpha=\frac{1}{p!}\alpha_{r_1 \cdots r_p}dz^{r_1}\wedge \dots \wedge dz^{r_p} $ and $(g^{r\bar s})$ is the inverse of the Hermitian matrix $(g_{r\bar s})$ induced by the $I$-Hermitian metric $g$.

We will need the following Hodge star-type operator $*\colon \Lambda_I^{p,0}(M) \to \Lambda_I^{2n-p,0}(M)$, defined by the relation
\[
\alpha \wedge *\beta=\frac{1}{n!} \langle \alpha,\beta\rangle_{g}\Omega_0^n\,,\qquad \text{for } \alpha,\beta \in \Lambda_I^{p,0}(M)\,.
\]

We fix a point $x_0\in M$ and take holomorphic coordinates $(z^1,\dots,z^{2n})$ with respect to $I$ such that $(g_{r\bar s})$ is the identity at $x_0$, then we may compute
%\[
%*dz^i=(-1)^{i-1}dz^1\wedge \cdots \wedge \widehat{dz^i} \wedge \cdots \wedge dz^{2n}
%\]
%and also
\begin{equation}
\label{Hodge}
*(dz^{2i-1}\wedge dz^{2i})=dz^1\wedge \cdots \wedge \widehat{dz^{2i-1}} \wedge\widehat{dz^{2i}} \wedge \cdots \wedge dz^{2n}\,.
\end{equation}

Observe that the Hodge operator sends q-real $(2,0)$-forms to q-real $(2n-2,0)$-forms and vice versa. Recall that, when the hypercomplex structure is locally flat, to any q-real $(2,0)$-form $\Omega$ is associated a hyperhermitian matrix $(\Omega_{\bar r s})$, thus, we may define the determinant of $\Omega$ as the Moore determinant of $(\Omega_{\bar r s})$. This definition naturally extends to any q-real $(2n-2,0)$-form $\Phi$ by setting $\det(\Phi)=\frac{1}{(n-1)!}\det(*\Phi)$. In particular, for any q-real $\Omega \in \Lambda_I^{2,0}(M)$, we have
\begin{equation}\label{Hodge1}
\det(\Omega^{n-1})=\det(\Omega)^{n-1}\,,
\end{equation}
which can be checked by taking coordinates in which $(\Omega_{\bar r s})$ is diagonal at a given point and using \eqref{Hodge}. Indeed, the fact that we can choose coordinates that diagonalize both $(g_{r\bar s})$ and $(\Omega_{\bar r s})$ is ensured by \cite[Lemma 3]{Sroka}. For any pair of q-real $\chi, \Omega \in \Lambda_I^{2,0}(M)$, we also have 
\begin{equation}\label{Hodge2}
\frac{\chi^n}{\Omega^n}=\frac{\det(\chi)}{\det(\Omega)}=\frac{\det(*\chi)}{\det(*\Omega)}\,.
\end{equation}

A q-real $(2n-2,0)$-form $\Phi$ is said to be positive if $\Phi \wedge \Omega>0$ for all positive $(2,0)$-forms $\Omega$. We observe that the Hodge star maps positive $(2,0)$-forms to positive $(2n-2,0)$-forms and conversely. On a locally flat hyperhermitian manifold the $(n-1)$\textsuperscript{th} power $\Omega\mapsto \Omega^{n-1}$ is a bijective correspondence between the cone of positive $(2,0)$-forms and the cone of positive $(2n-2,0)$-forms. The proof of this fact is just a matter of linear algebra and it is entirely analogous to the argument in \cite[pp. 279-280]{Michelson}, therefore we omit it.

\begin{proof}[Proof of Corollary $\ref{cor}$] 
For starters, we claim 
\begin{equation}\label{Hodge3}
	\frac{1}{(n-1)!} * \left( \partial \partial_J \phi \wedge \Omega_0^{n-2} \right)=\frac{1}{n-1} \left[ (\Delta_g \phi) \Omega_{0}-\partial \partial_J \phi \right]\,,
\end{equation}	
for any arbitrary function $\phi \in C^2(M,\R)$. It is enough to prove that for every $W\in \Lambda_I^{2n-2,0}(M)$, we have
\[
\partial \partial_J \phi \wedge \frac{\Omega_0^{n-2}}{(n-2)!}\wedge (*W)=(\Delta_g \phi) W\wedge\Omega_{0}-W\wedge \partial \partial_J \phi.
\]
Let $Z=dz^{1}\wedge \cdots \wedge dz^{2n}$ for simplicity and fix a point $x_0\in M$ where $\Omega_0$ takes the standard form
\[
\Omega_0=\sum_{i=1}^n dz^{2i-1}\wedge dz^{2i}\,.
\]
Without loss of generality, we may assume 
$W=\widehat{dz^{1}}\wedge \widehat{dz^{2}}\wedge dz^{3}\wedge \cdots \wedge dz^{2n}$. It is easy to see that
\[
W\wedge\Omega_{0}=Z\,, \qquad W\wedge \partial \partial_J \phi=(\phi_{1\bar 1}+\phi_{2\bar 2})Z\,.
\]
As $*W=dz^{1}\wedge dz^{2}$, we obtain
\[
\begin{split}
\partial \partial_J \phi \wedge \frac{\Omega_0^{n-2}}{(n-2)!}\wedge (*W)=&
\partial \partial_J \phi \wedge \frac{\Omega_0^{n-2}}{(n-2)!}dz^{1}\wedge dz^{2}\\
=&\partial \partial_J \phi \wedge \sum_{i>1} dz^{1}\wedge dz^{2}\wedge \cdots \widehat{dz^{2i-1}}\wedge \widehat{dz^{2i}} \wedge \cdots \wedge dz^{2n}\\
=&\sum_{i>1}(\phi_{2i-1\overline{2i-1}}+\phi_{2i\overline{2i}})Z=(\Delta_g\phi)Z-(\phi_{1\bar 1}+\phi_{2\bar 2})Z\\
=&(\Delta_g \phi) W\wedge\Omega_{0}-W\wedge \partial \partial_J \phi\,,
\end{split}
\]
as claimed.

From \eqref{Hodge2} and \eqref{Hodge3}, it follows that
\[
\begin{split}
\frac{\left(\Omega_1+\frac{1}{n-1} \left[(\Delta_g \phi) \Omega-\partial \partial_J \phi \right]\right)^n}{\Omega_0^{n}}&=\frac{\det\left(*\left(\Omega_1+\frac{1}{n-1} \left[(\Delta_g \phi) \Omega-\partial \partial_J \phi \right]\right)\right)}{\det(*\Omega_0)}\\
&=\frac{\det\left(\Omega_{2}^{n-1}+\partial \partial_J \phi \wedge \Omega_0^{n-2}\right)}{\det(\Omega_0^{n-1})}\,.
\end{split}
\]
This implies that given a positive $(2,0)$-form $\Omega_1$ and a smooth function $H$ on $M$, the pair $(\phi,b)\in C^\infty(M,\R)\times \R_+$ is a solution to \eqref{n-1 MA equation 1} if and only if it solves
\begin{equation}\label{n-1 MA equation 2}
	\begin{cases}
		\det (\Omega_{2}^{n-1}+\partial \partial_J \phi \wedge \Omega_0^{n-2})=b\,{\rm e}^{H}\det(\Omega_{0}^{n-1})\,, \\[2mm]
		\Omega_{2}^{n-1}+\partial \partial_J \phi \wedge \Omega_0^{n-2} > 0\,, \quad \sup_{M}\phi = 0\,,
	\end{cases}
\end{equation}
where $\Omega_2$ is uniquely defined by
\begin{equation*}
\Omega_1=\frac{1}{(n-1)!}*\Omega_2^{n-1}\,,
\end{equation*}
because the $(n-1)$\textsuperscript{th} power is a bijection between the spaces of positive $(2,0)$-forms and positive $(2n-2,0)$-forms.

Now, let $(\phi,b)\in C^\infty(M,\R)\times \R_+$ be the solution to \eqref{n-1 MA equation 1}, or equivalently \eqref{n-1 MA equation 2}, with datum $H=(n-1)H'$. Define $\tilde \Omega$ as the unique $(n-1)$\textsuperscript{th} root of $\Omega_{2}^{n-1}+\partial \partial_J \phi \wedge \Omega_0^{n-2}$. Then it is clear that if $\Omega_2$ is the $(2,0)$-form induced by a quaternionic balanced (resp. quaternionic Gauduchon, quaternionic strongly Gauduchon) metric, then so is $\tilde \Omega$. Finally, set $b'=b^{1/(n-1)}$, then using \eqref{Hodge1} we conclude
\[
\frac{\tilde\Omega^n}{\Omega_0^n} %=\frac{\det(\tilde \Omega)}{\det(\Omega_0)}
=\left(\frac{\det(\tilde \Omega^{n-1})}{\det(\Omega_0^{n-1})}\right)^{\frac{1}{n-1}}=\left( \frac{\det \left( \Omega_{2}^{n-1}+\partial \partial_J \phi \wedge \Omega_0^{n-2}\right)}{\det(\Omega_0^{n-1}) } \right)^{\frac{1}{n-1}}=\left( b\, {\rm e}^H \right)^{\frac{1}{n-1}}=b'\, {\rm e}^{H'}\,. \qedhere
\]
\end{proof}

\end{document}